\documentclass[11pt,letterpaper]{article}

\usepackage[utf8]{inputenc}
\usepackage[T1]{fontenc}
\usepackage{lmodern}
\usepackage{microtype}
\usepackage[margin=1in]{geometry}

\usepackage{amsmath,amssymb,amsthm,amsfonts,mathtools}
\usepackage{mathrsfs}
\usepackage{bm}
\usepackage[colorlinks=true,linkcolor=blue!70!black,citecolor=green!50!black,urlcolor=blue!70!black]{hyperref}
\usepackage[nameinlink,capitalize]{cleveref}
\usepackage{enumitem}
\usepackage{algorithm}
\usepackage{algorithmic}
\usepackage{xcolor}

\theoremstyle{plain}
\newtheorem{theorem}{Theorem}[section]
\newtheorem{lemma}[theorem]{Lemma}
\newtheorem{proposition}[theorem]{Proposition}
\newtheorem{corollary}[theorem]{Corollary}

\theoremstyle{definition}
\newtheorem{definition}[theorem]{Definition}

\theoremstyle{remark}
\newtheorem{remark}[theorem]{Remark}

\newcommand{\E}{\mathbb{E}}

\newcommand{\R}{\mathbb{R}}

\newcommand{\cD}{\mathcal{D}}

\newcommand{\Var}{\mathrm{Var}}

\newcommand{\norm}[1]{\left\lVert #1\right\rVert}

\newcommand{\argmin}{\mathop{\mathrm{argmin}}}

\newcommand{\Hamm}{d_H}

\DeclareMathOperator{\Bin}{Bin}

\DeclareMathOperator{\Unif}{Unif}

\newcommand{\ES}{\mathsf{ES}}

\allowdisplaybreaks

\newenvironment{proofsketch}{\noindent {\em {Proof sketch.}}}{$\blacksquare$\vskip \belowdisplayskip}

\newcounter{constant}

\newcommand{\newconstant}[1]{%
  \begingroup
    \refstepcounter{constant}%
    \label{#1}%
  \endgroup
}

\newcommand{\useconstant}[1]{C_{\ref{#1}}}

\title{Robust Statistical Estimators with Bounded Empirical Sensitivity\thanks{Authors are listed in alphabetical order.}}
\author{Valentio Iverson\thanks{University of Waterloo. \texttt{viverson@uwaterloo.ca}. Supported by the University of Waterloo through an MURA and a URF project.} \and Gautam Kamath\thanks{University of Waterloo and Vector Institute. \texttt{g@csail.mit.edu}. Supported by a Canada CIFAR AI Chair, an NSERC Discovery Grant, and an Ontario Early Researcher Award.} \and Argyris Mouzakis\thanks{University of Waterloo. \texttt{amouzaki@uwaterloo.ca}. Supported by an Ontario Early Researcher Award and a David R. Cheriton Graduate Scholarship.} \and Adam Smith\thanks{Boston University. 
\texttt{ads22@bu.edu}. Supported in part by US NSF award 2232694 and a gift from Apple Research.}}
\date{\today}

\begin{document}
\maketitle

\begin{abstract}
We introduce a new measure of robustness for statistical estimators, which we call \emph{empirical sensitivity}.
An estimator $\hat \theta$ has bounded empirical sensitivity if, with high probability over a dataset $X = (X_1, \dots, X_n) \sim \mathcal{D}^{\otimes n}$, for any dataset $Y$ obtained by modifying at most $\eta n$ points in $X$, we have that $\hat \theta(Y)$ is close to $\hat \theta(X)$. 

We study bounds on this quantity for the prototypical problem of Gaussian mean estimation. 
We prove new lower bounds, showing that for any estimator $\hat \mu$ which achieves an optimal $\ell_2$-error bound of $O(\sqrt{d/n})$, the empirical sensitivity is at least $\Omega(\eta + \sqrt{\eta d/n})$. 
The two terms arise due to obstructions on the mean and variance (via an Efron-Stein argument) of such an estimator.
We show that this bound is tight up to logarithmic factors, by employing recent results for robust empirical mean estimation.

\end{abstract}

\setcounter{tocdepth}{1}

\section{Introduction}
Robustness is a classic topic of study within Statistics~\cite{HuberR09}.
It tries to understand and bound how much an estimator can fluctuate when data diverges from our assumptions. 
Such divergences could arise for innocuous reasons such as model misspecification, or more malicious reasons like an adversary manipulating or poisoning the dataset. 

A popular style of robustness ensures that an estimator maintains accuracy after an adversary modifies a portion of the input dataset.
In more detail, suppose we have an estimator $\hat \theta$ for a parameter $\theta^*$ of some family of distributions $\mathcal{D}_\theta$. 
Let $X = (X_1, \dots, X_n) \sim \mathcal{D}_{\theta^*}$, and let $Y$ be any dataset obtained by modifying $\eta n$ points in $X$. 
Robustness of the estimator $\hat \theta$ is then measured in terms of $\|\hat \theta(Y) - \theta^* \|$, the distance between the estimator on the contaminated dataset and the true parameter of interest. 

While algorithms satisfying this definition are quite powerful, the notion does not capture some natural properties we might hope for an estimator to possess.
We illustrate by considering the estimator $\hat \mu$ to be the median, and the dataset $X$ to be sampled from a univariate Gaussian distribution $\mathcal{N}(\mu, 1)^{\otimes n}$.
For an uncontaminated dataset $X$, with high probability, the median is well known to achieve the near optimal statistical rate: $|\hat \mu(X) - \mu| \leq O(1/\sqrt{n})$.
Furthermore, the median is the canonical example of a robust statistic.
For any dataset $Y$, formed by changing $\eta n$ entries of $X$, we have that $|\hat \mu(Y) - \mu| \leq O(\eta + 1/\sqrt{n})$---a mild (but inherent) degradation of the rate from the uncontaminated case.
However, we consider how much the estimator can shift due to contamination. 
For simplicity, focus on the case when $\eta = 1/n$, i.e., a single point is contaminated. 
The aforementioned robustness guarantee allows us only to say that $|\hat \mu(Y) - \hat \mu(X)| \leq O(1/n + 1/\sqrt{n})$.
This is a rather large radius for only a single contamination, with the dominant term being the parametric rate of $1/\sqrt{n}$. 
In fact, as we show in Section~\ref{sec:emp-median}, the median enjoys a quadratically stronger guarantee on this quantity: $|\hat \mu(Y) - \hat \mu(X)| \leq O(1/n)$. 
To summarize, the conventional measures of robustness may not give insight into how \emph{sensitive} the estimator may be. 

Motivated by this deficiency, we introduce a new measure of robustness for statistical estimators, which we call the \emph{empirical sensitivity}. 
Informally, we would like for an estimator $\hat \theta$ to simultaneously have the following two properties.
\begin{itemize}
    \item Accuracy: Given a dataset $X \sim \mathcal{D}_{\theta^*}$, we have that $\|\hat \theta(X) - \theta^*\|$ is small.
    \item Bounded empirical sensitivity:
    For $X \sim \mathcal{D}_{\theta}$, and for any dataset $Y$ obtained by modifying $\eta n$ entries of $X$, we have that $\|\hat \theta(Y) - \hat \theta(X)\|$ is small---either in expectation over $X$ or with high probability.
\end{itemize}
Similar to traditional robustness, we focus on the case where the dataset $X$ is drawn stochastically, but $Y$ is an arbitrary or worst-case contamination of $X$.
As usual, obtainable guarantees will depend on properties of the class of distributions $\mathcal{D}_\theta$. 
On the other hand, this can be seen as a \emph{stronger} type of robustness: by triangle inequality, an accurate estimator with bounded empirical sensitivity will also enjoy a bound on $\|\hat \theta(Y) - \theta^* \|$, the traditional object of study in robust estimation.

The robustness of the estimated value to corruption has been considered before as a \emph{worst-case} constraint on a function, notably in the design of differentially private algorithms. To our knowledge, ours is the first work that studies the interaction of empirical sensitivity with statistical efficiency, and also the first to explicitly study the average-case notion. See Related Work (\Cref{sec:related}) for more detail.

\subsection{Results and Techniques}

We initiate study of empirical sensitivity focusing on one of the most fundamental statistical tasks: \emph{Gaussian mean estimation}. 
Our main result is the following lower bound on the empirical sensitivity of optimal estimators.

\begin{theorem}
\label{thm:main-theorem}
Consider any estimator $\hat \mu$ such that  
\begin{equation} \label{eq:ideal-MSE}
\sup_{\mu \in \mathbb{R}^d} \mathbb{E}_{X \sim \mathcal{N}(\mu, \mathbb{I}_d)^{\otimes n}} \| \hat \mu(X) - \mu \|_2^2 \lesssim \frac{d}{n}. 
\end{equation} 

For every $\eta \in (0,1)$,\footnote{Throughout, it is implicit that $\eta \geq 1/n$, i.e., there is at least one contaminated point.} we have that 
\[ \sup_{\mu \in \mathbb{R}^d} \left( \mathbb{E}_{X \sim \mathcal{N}(\mu, \mathbb{I}_d)^{\otimes n}} \left[ \sup_{Y \colon d_H(X,Y) \le \lfloor \eta n \rfloor} \| \hat \mu(Y) - \hat \mu(X) \|_2^2 \right] \right)^{1/2} \ge \Omega \left( \eta + \sqrt{\frac{\eta d}{n}} \right).\]
\end{theorem}

This lower bound is tight up to logarithmic factors.
Very recent work by Chen, Ding, Majid, and McKelvie~\cite{ChenDMM26} studies robust \emph{empirical} mean estimation. 
An easy argument shows that their estimator satisfies near-optimal accuracy and also average-case
empirical sensitivity (Theorem~\ref{thm:upper-main}) essentially matching \Cref{thm:main-theorem} above. 

We recall that, for robust Gaussian mean estimation, the  optimal rate for $\|\hat \mu(X) - \mu\|_2$ is $\tilde \Theta\left(\eta + \sqrt{\frac{d}{n}} \right)$.
By triangle inequality, this implies that one may have hoped for an empirical stability bound as small as $\tilde \Theta(\eta)$.\footnote{Note that triangle inequality only implies a lower bound of $\Omega(\eta)$ in the regime where $\eta = \Omega\left(\sqrt{d/n}\right)$. 
Indeed, proving the lower bound of $\eta$ in all regimes appears to require significant technical work, and is quite different from how this term arises for conventional robust statistics, see Section~\ref{sec:main-results} for more details.}
Instead, the larger bound of $\tilde \Theta\left(\eta + \sqrt{\frac{\eta d}{n}}\right)$ arises. 

Technically, the two terms in the lower bound are proven by showing obstructions on the mean and variance of the estimator. 

\paragraph{The mean obstruction.} We first reduce a $d$-dimensional uniformly accurate estimator to a scalar estimator along a direction and an orthogonal subspace. Then we prove scalar sensitivity lower bounds in two overlapping regimes. In the low corruption regime $\eta \lesssim \frac{1}{\sqrt{n}}$, we compare a valid local corruption, which shifts $k$ random samples, to a statistically indistinguishable global dataset corruption that the estimator is forced to track due to its accuracy guarantees.  For $\eta \gtrsim \frac{\log n}{n}$, a coupling argument compares two samples from two nearby Gaussian means whose total variance distance is of order $\eta$, so that the coupled samples differ in at most $\lfloor \eta n \rfloor$ coordinates with high probability. The two regimes overlap for large $n$, covering the full range of values for $\eta$.\footnote{In Section~\ref{sec:bernoulli}, we also show this lower bound for the conceptually simpler case of Bernoulli parameter estimation.} 

\paragraph{The variance obstruction.} The $\sqrt{\eta d/n}$ lower bound is driven by clean-sample variability. Indeed, we show that any uniformly MSE-accurate estimator must have output variance $\Omega(d/n)$ at some parameter. Using this, a block-resampling argument and vector Efron-Stein inequality convert this clean output variance into sensitivity to replacing an $\eta$ fraction of the samples.

\subsection{Related Work}
\label{sec:related}

Robust estimation is a vibrant area of study, particularly recent work with a focus on minimax statistical rates and computational efficiency in multivariate settings, see, e.g.,~\cite{DiakonikolasKKLMS16,LaiRV16} and~\cite{DiakonikolasK22} for a textbook treatment. 
As already described, empirical sensitivity is stronger than 
what is implied by traditional robustness guarantees.
One recent work of Chen, Ding, Majid, and McKelvie~\cite{ChenDMM26} studies differentially private Bayesian estimation, which they reduce to robust estimation of an empirical mean.
While not originally conceived as such, their robust algorithm provides an estimator with bounded empirical sensitivity for a particular statistic that satisfies the desired accuracy guarantee (i.e., the empirical mean). 

Stability has long been an object of study in statistics and machine learning, see, e.g.,~\cite{KearnsR97,BousquetE00}. 
This generally focuses on understanding how the error 
of an algorithm changes if one removes a single point.
Our focus instead is on the effect of modifying a \emph{collection} of points (and also on the entire output, not just the error).  
The recent, contemporaneous work of Chakraborty, Luo, and Barber~\cite{ChakrabortyLB26} explores minimax risk when an estimator must 
have low \textit{global sensitivity}---that is, \textit{worst-case} empirical robustness with respect to a single corruption. 

They suggest study of an distributional variant, where the dataset is sampled from a distribution belonging to some family, as a direction for future work. 

Empirical stability also relates to concepts considered in the literature on differential privacy (DP)~\cite{DworkMNS06}. 
Many DP mechanisms operate by adding noise calibrated to the worst-case sensitivity (called the \textit{global sensitivity}) of a function. 
(This differs from our notion in that it is for worst-case neighboring datasets and considers only a single corruption.) 
The literature also considers the \emph{local sensitivity}~\cite{NissimRS07}: the sensitivity of the function at a particular dataset. 
In particular, our definition can be viewed as an average-case bound on
\emph{local sensitivity with step size $s$} (Definition 4.4 of the full version of~\cite{NissimRS07}) for $s=\eta n$.

A similar notion is the \textit{inverse sensitivity} at a data set $X$, employed in the inverse sensitivity mechanism~\cite{MirMNW11, JohnsonS13,AsiD20a, AsiD20b}. This measures how many corruptions to $X$ would be needed to change the value of an estimator from $f(X)$ to a given target value $z$. 
Empirical sensitivity provides an average-case bound on this quantity. That connection was used implicitly in some works relating robustness to privacy, such as \cite{HopkinsKMN23,AsiUZ23}.

Several other ``local'' variants of sensitivity have been considered, for example \textit{down sensitivity}~\cite{ChenZ13,RaskhodnikovaS15} (related to \textit{resilience} in robust statistics \cite{SteinhardtCV18}), which only considers neighboring datasets where points have been \textit{removed}.  
We believe the full extent of connections between empirical sensitivity and DP are yet to be understood.

Finally, the recent work of Trillos, Jaffe, and Sen~\cite{TrillosJS25} introduces a new notion of sensitivity, which measures change in a statistic when infinitesimal Gaussian noise is added to each datapoint. 
Similar to our notion, it is an average-case (as opposed to worst-case) quantity, but the different contamination models appear to capture different phenomena. 
In particular, their notion may connect more closely with the literature on measurement error and local differential privacy~\cite{Warner65,EvfimievskiGS03,KasiviswanathanLNRS11}.

\section{Problem Setup and Definitions}

\textbf{General Notation.}
We use the notation $A\lesssim B$ to mean that $A\le C B$ for a universal constant $C>0$, and $A\gtrsim B$ analogously. We write $A\asymp B$ if both $A\lesssim B$ and $B\lesssim A$ hold. We also note that $\widetilde O(\cdot)$ suppresses polylogarithmic factors in the relevant problem parameters. Throughout the whole of this work, we assume that $d, n \ge 1$ are fixed, while all norms are the standard Euclidean norms.

For two datasets $x,x'\in(\R^d)^n$, let $\Hamm(x, x') \coloneqq \# \{ i \in [n] \colon x_i \ne x_i' \}$ denote the Hamming distance at the sample level. Thus, $\Hamm(x, x') \le k$ means that $x'$ can be obtained from $x$ by replacing at most $k$ samples by arbitrary vectors in $\R^d$. Given a sample space $\mathcal{X}$, we denote the set of all distributions over $\mathcal{X}$ by $\Delta(\mathcal{X})$. For any distribution $\mathcal{D} \in \Delta(\mathcal{X})$, $\mathcal{D}^{\otimes n}$ denotes the product distribution over $\mathcal{X}^n$, where each marginal is $\mathcal{D}$. Given a dataset $X = (X_1, \dots, X_n)$ drawn i.i.d.\ from $\mathcal{D}$, we write $X \sim \mathcal{D}^{\otimes n}$. Whenever using the symbols of probability and expectation, we use a subscript to denote what the randomness is over in cases where it might not be clear from the context, e.g., $\mathbb{E}_X[\cdot]$. 

For $\eta \in (0,1)$, we consider the $\eta$-corruption model where the adversary is allowed to contaminate any of $k \coloneqq \lfloor \eta n \rfloor$ points from the given $n$ sample points.  

\begin{definition}[Uniform MSE Accuracy]
\label{def:mse}
Let $\cD$ be a distribution class parameterized by $\theta \in \Theta$, and let $D_{\theta} \in \cD$.
We assume that $f$ is an estimator that takes $n$ samples drawn i.i.d.\ from $D_{\theta}$ and outputs an estimate of $\theta$.
We say that $f$ is \emph{$(C_{\mathsf{MSE}}, \gamma)$-uniformly MSE-accurate} if
\begin{equation}\label{eq:clean-mse-main}
\sup_{\theta\in\Theta}\E_{X \sim D_{\theta}^{\otimes n}}\left[\norm{f(X)-\theta}^2\right] \le C_{\mathsf{MSE}} \gamma.
\end{equation}
\end{definition}

In the above definition, the error rate includes two parameters. These are $C_{\mathsf{MSE}}$ and $\gamma$. To give a bit more context about this, $C_{\mathsf{MSE}}$ should be thought of as an absolute constant, whereas $\gamma$ should be treated as a component of the rate that may depend on various aspects of the class of distributions under consideration (e.g., the dimension of the data), as well as the number of samples used by our estimator.

\begin{definition}[Pointwise empirical sensitivity \cite{NissimRS07}]
\label{def:pointwise-sens}
For $\eta \in (0,1)$, we define the $\eta$-corruption pointwise empirical sensitivity of an estimator $f$ at dataset $X$ as: 
\[ S_\eta^f(X) \coloneqq \sup_{X'\colon \Hamm(X,X')\le \lfloor \eta n \rfloor}\norm{f(X)-f(X')}. \]
\end{definition}

The pointwise definition of empirical sensitivity was considered previously, under the name \textit{local sensitivity at step size $\lfloor\eta n\rfloor$}.

\begin{definition}[Distributional empirical sensitivity]
\label{def:dist-sens}
Let $\cD$ be a family of distributions parameterized by $\theta \in \Theta$.
For $D_{\theta} \in \cD, q\in\{1,2\}$, and $\eta \in (0,1)$, we define the $\eta$-corruption $L^q$-distributional empirical sensitivity of an estimator $f$ at distribution $D_{\theta}$ as 
\[ \ES_{\eta,q}(f;\theta) \coloneqq \biggl(\E_{X \sim D_{\theta}^{\otimes n}}\bigl[(S_\eta^f(X))^q\bigr]\biggr)^{1/q}. \]
Thus $\ES_{\eta,1}$ is expected pointwise sensitivity and $\ES_{\eta,2}$ is root-mean-square pointwise sensitivity.
\end{definition}

\subsection{Adversary Models}
\label{subsec:adv}

The robustness literature frequently studies the relative power of different adversaries~\cite{BlancLMT22,BlancV25,BlancHMS26,BenDavidBKL23,LechnerBK25}. 
We briefly discuss various adversary models, listed here in increasing amount of strength.


\begin{definition}
\label{def:three-adversaries}
Let a dataset $X = (X_1, \dots, X_n) \sim \mathcal{D}^{\otimes n}$, where $\mathcal{D}$ is a probability distribution. 
Fix a corruption fraction $\eta \in (0,1)$ and define $k \coloneqq \lfloor \eta n \rfloor$ to be the number of contaminated points.\footnote{One could consider broader classes of adversaries.
For example, an adversary who adaptively chooses $k$ points adversarially, and then resamples them obliviously. Or an adversary who obliviously chooses $k$ points and then replaces them arbitrarily. We omit discussion of such adversaries to streamline the presentation.}
\begin{enumerate}
    \item \emph{Resampling adversary}. A subset $I \subseteq [n]$ of size $k$ is chosen uniformly at random, and the samples $\{X_i \colon i \in I\}$ are replaced by fresh independent samples from the same clean distribution as $X$. 
    This is the weakest adversary: both the removed points and the inserted points are stochastic and drawn from the clean distribution.

    \item \emph{Stochastic adversary}. A subset $I \subseteq [n]$ of size $k$ is chosen uniformly at random, and the samples $\{X_i \colon i \in I\}$ are replaced by independent samples from a fixed replacement distribution $\mathcal{Q}$ chosen before seeing $X$. The replacement distribution may differ from the clean distribution, but it is not chosen adaptively after observing the realized dataset.

    \item \emph{Adaptive adversary}. After observing the clean dataset $X$, the adversary may choose both the corrupted indices and the replacement values arbitrarily, subject only to the constraint $d_H(X, Y) \le k$. This is the adversary encoded by $S_\eta^f(X)$.
\end{enumerate}
\end{definition}

Our discussion thus far has focused on the adaptive adversary. 
This is reflected in our main result.
The result focuses on the class of Gaussian distributions with identity covariance $\{\mathcal{N}(\mu, \mathbb{I}_d)\}_{\mu \in \R^d}$, and shows a lower bound on 
\[
    \sup_{\mu \in \R^d} \ES_{\eta,2}(f; \mu) 
    = \sup_{\mu \in \R^d} \left( \E_{X \sim \mathcal{N}(\mu, \mathbb{I}_d)^{\otimes n}} \bigl[ (S_\eta^f(X))^2 \bigr] \right)^{1/2}.
\]

The two terms in the lower bound of Theorem~\ref{thm:main-theorem} come from different adversaries. 
The variance-obstruction term of $\Omega(\sqrt{\eta d/n})$ arises even for the weakest adversary, the resampling adversary. 
On the other hand, our proof for the mean-obstruction term of $\Omega(\eta)$ requires a stronger adversary. More precisely, the low-corruption proof uses a random-subset mean-shift adversary. While this mean-shift adversary and the standard stochastic adversary are technically distinct -- as they induce different joint distributions between the clean and corrupted samples -- they are identical in their marginal distributions. Because our argument relies solely on these marginal distributions, the proof applies directly to both the mean-shift and stochastic adversaries. By contrast, the high-corruption proof uses a coordinatewise TV-coupling adversary, where each sample is coupled independently and the corrupted indices are those coordinates on which the coupling fails; this is adaptive to the realized sample, but only coordinate-by-coordinate. As we show in Section~\ref{sec:resampling-adversary}, more empirically stable estimators exist under the (weakest) resampling adversary. 
In this case, the empirical mean achieves an empirical stability of $O(\sqrt{\eta d/n})$, avoiding the $\Omega(\eta)$ term that is necessary under the adaptive adversary.

As the empirical sensitivity is a quantity measured with respect to a realized dataset, all our adversaries are defined as modifying a dataset, rather than a distribution (as is sometimes done in the conventional robustness literature). 
Our adaptive adversary is the same as what is usually called the \emph{strong contamination} model in the robustness literature~\cite{DiakonikolasK22}. 
The stochastic adversary can be seen as the moral equivalent of Huber contamination~\cite{Huber64}.
The resampling adversary has no equivalent, as, by construction, the underlying distribution is uncontaminated.

\section{Main Results}
\label{sec:main-results}

\newconstant{const:upperMSE}
\newconstant{const:mainLower}
\newconstant{const:meanlowLower}
\newconstant{const:meanhighLower}
\newconstant{const:varLower}

In this section, we give the bulk of our results. We will be working with the class of Gaussian distributions with unknown mean $\mu$ and known covariance matrix $\Sigma$ which, without loss of generality, we assume to be equal to the identity matrix $\mathbb{I}$. We will consider mean estimators with bounded empirical sensitivity for this class, and give both upper and lower bounds.

The following theorem is the main result of this paper: a tight lower bound of the empirical sensitivity of accurate estimators, complemented by a recent upper bound result in \cite{ChenDMM26}:

\begin{theorem} \label{thm:main}
Let $f \colon (\mathbb{R}^d)^n \to \mathbb{R}^d$ be measurable and $\left(C_{\mathsf{MSE}}, \frac{d}{n}\right)$-uniformly MSE-accurate for the class of Gaussians $\{\mathcal{N}(\mu, \mathbb{I})\}_{\mu \in \R^d}$. Then for every $\eta \in (0,1)$ and all sufficiently large $n$ depending only on $C_{\mathsf{MSE}}$, we have
    \[ \sup_{\mu \in \mathbb{R}^d} \ES_{\eta, 2}(f; \mu) \ge \Omega_{C_{\mathsf{MSE}}}  \left( \eta + \sqrt{\frac{\eta d}{n}} \right). \]
\end{theorem}

The lower bound comprises two independent components: a mean obstruction, which gives the $\eta$ term, and a variance obstruction, which gives the $\sqrt{\eta d/n}$ term. Each part captures the estimator's vulnerability to a different adversarial mechanism. The mean obstruction reflects the estimator's sensitivity to small shifts in the underlying distribution's mean, whereas the variance obstruction isolates the estimator's inherent variability when resampled from the same clean distribution. 

\subsection{Upper bound from robust empirical-mean recovery}
Our lower bound is complemented by a high-probability upper bound, which is a sensitivity consequence of an estimator that remains close to the clean empirical mean simultaneously over all allowed corruptions. 
The following is a consequence of a result of Chen, Ding, Majid, and McKelvie~\cite{ChenDMM26}.

\begin{theorem}
\label{thm:upper-main} 
Fix $\eta \in (0,1/3)$ and $\beta \in (0,1/2)$, there exists a measurable estimator $f_{\eta, \beta} : (\mathbb{R}^d)^n \to \mathbb{R}^d$ such that, for every $\mu \in \mathbb{R}^d$, with probability at least $1 - \beta$ over $X \sim \mathcal{N}(\mu, \mathbb{I}_d)^{\otimes n}$, both of the following hold:
\begin{align*} 
\| f_{\eta, \beta}(X) - \mu \|_2^2 &\le O \left( \frac{d + \log(1/\beta)}{n} \right) \quad \text{and} \quad S_{\eta}^{f_{\eta, \beta}}(X) \le \widetilde{O} \left( \eta + \sqrt{\eta \cdot \frac{d + \log(1/\beta)}{n}} \right)
\end{align*} 
\end{theorem}

\begin{proofsketch}
The primitive from \cite{ChenDMM26} gives a high-probability event on which, simultaneously for every $\eta$-corruption $X'$ of the clean sample $X$, every feasible cleaned candidate has empirical mean close to the empirical mean $\overline{X}$. We modify the estimator so that, on clean input, it chooses the feasible candidate with minimum Hamming distance to the input. Therefore, one can show that on the same good event, $f_{\eta, \beta}(X) = \overline{X}$. The complete proof can be found in \cref{app:upper-main}.
\end{proofsketch}

\begin{remark}[Efficient versus inefficient upper bounds]
While the estimator of Theorem~\ref{thm:upper-main} should be viewed as the statistically sharp upper bound needed to match our lower bound up to logarithmic factors, it is not computationally efficient.
The same work also gives an efficient estimator with a worse empirical-mean recovery rate. 
Applying the same triangle-inequality argument as above yields an efficient estimator for our empirical-sensitivity problem with sensitivity of order $\widetilde{O} \left( \eta + \sqrt{\eta \cdot \sqrt{\frac{d}{n}}} \right)$ with high probability, together with clean-data MSE accuracy. 

Their efficient estimator is complemented by a computational lower bound against restricted class of algorithms.
Informally, they show that one can not do better than this weaker $O\left(\sqrt{\eta \sqrt{\frac{d}{n}}}\right)$ rate with any algorithm implemented by low-degree polynomials.
We note that this does \emph{not} imply a computational lower bound for our setting, as their setting is specific to the empirical mean, whereas the problem we are considering allows any estimator which is MSE accurate.

It remains open whether efficient, uniformly accurate estimators can achieve the statistically optimal empirical-sensitivity scale $\eta + \sqrt{\frac{\eta d}{n}}$, or whether the worse efficient rate $\eta + \sqrt{\eta \sqrt{\frac{ d}{n}}}$ is inherent for computationally-efficient estimators under our empirical sensitivity condition. 
\end{remark}

\subsection{Mean obstruction}
We establish the mean obstruction by separating the analysis into two distinct regimes that collectively cover the entire spectrum of corruption levels. In the low-corruption regime, where $\eta = O(1/\sqrt{n})$, our argument relies on bounding the $\chi^2$-divergence between two distinct adversaries: a local shift of exactly $k$ random samples by a constant $\delta$ (which represents a valid adversarial $\eta$-contamination), and a global shift of all samples by $\eta \delta$ (which represents a genuine change in the underlying parameter). Because the $\chi^2$-divergence between these two distributions scales as $O(k^2\delta^4/n)$, choosing a sufficiently small $\delta$ renders the adversarial contamination statistically indistinguishable from a true parameter shift. Since any MSE-accurate estimator is mathematically forced to track the global shift to maintain its low average risk, this indistinguishability ensures the estimator is similarly displaced by the local corruption, yielding the $\Omega(\eta)$ lower bound. 

For the high-corruption regime, where $\eta = \Omega(\log n/n)$, we shift to a coordinate-wise maximal coupling argument. We couple two Gaussian distributions whose means are separated by exactly $\eta$. By showing that samples drawn from these coupled distributions differ in at most $k$ coordinates with high probability in this regime, we show that an $\eta$-bounded adversary can seamlessly bridge the two datasets. Since the estimator's accuracy forces its mean response to be displaced by $\Omega(\eta)$ between these two parameter values, the empirical sensitivity is similarly bounded below. Crucially, these two regimes are not isolated; they overlap for all sufficiently large $n$, ensuring that our combined lower bound fully covers the entire spectrum of corruption allowances without any gaps.

To prove both high-dimensional theorems, we reduce them to the single-dimensional case using a projection and conditioning lemma. Informally, any uniformly MSE-accurate estimator $f \colon (\mathbb R^d)^n\to\mathbb R^d$ can be projected along a carefully chosen direction to yield a one-dimensional estimator $g \colon \mathbb R^n\to\mathbb R$ with parameter $\mu'$ that inherits the $O(1/n)$ Bayes risk, allowing us to lower-bound the high-dimensional sensitivity $\ES_{\eta,2}(f;\mu)$ via the scalar empirical sensitivity $S_\eta^g(X)$.

\begin{theorem}[Mean obstruction: low-corruption regime]
\label{thm:mean-low}
Let $f \colon (\mathbb R^d)^n\to\mathbb R^d$ be measurable and $\left(C_{\mathsf{MSE}}, \frac{d}{n}\right)$-uniformly MSE-accurate for the class of Gaussians $\{\mathcal{N}(\mu, \mathbb{I})\}_{\mu \in \R^d}$. There exist constants $\kappa=\kappa(C_{\mathsf{MSE}})>0$ and $\useconstant{const:meanlowLower}=\useconstant{const:meanlowLower}(C_{\mathsf{MSE}})>0$ such that, for every $\eta \in (0,1)$ such that $1 \le \lfloor \eta n \rfloor \le \kappa \sqrt{n}$ and for all sufficiently large $n$, we have
\[ \sup_{\mu\in\mathbb R^d} \ES_{\eta,2}(f;\mu) \ge \useconstant{const:meanlowLower}\eta. \]
\end{theorem}

\begin{proofsketch}
We establish the low-corruption lower bound through an indistinguishability argument between two adversaries. As previously stated, we may assume that $\mu' \in [0,1]$ and thus it suffices to analyze a clipped scalar estimator $h : \mathbb{R}^n \to [0,1]$. 

\begin{enumerate}
    \item First, because the estimator inherits a low $O(1/n)$ integrated risk under a smooth prior, its expected output is forced to accurately track the true parameter. Specifically, if the entire data distribution is subjected to a global mean shift of $\eta \delta$, the estimator's mean output must naturally shift by approximately $\eta \delta$ to maintain its low average risk. We note that this global shift modifies \textit{all} $n$ points, meaning that it is \textit{not} a valid adversarial $\eta$-contamination. 
    \item Next, we consider a local shift: an adversary selects a uniformly random subset of exactly $k \coloneqq \lfloor \eta n \rfloor$ coordinates and shifts them by $\delta$.\footnote{In Section~\ref{subsec:adv} we defined the stochastic adversary as replacing part of the dataset with points drawn independently from another distribution. In the present case, the points are not resampled from scratch, but the $\delta$-shift results in a Gaussian distribution with different mean. While technically not exactly the same, this is arguably morally equivalent to a stochastic adversary.} Because this modifies at most $k$ samples, it \textit{is} a valid $\eta$-corruption. In the low-corruption regime $k = O(\sqrt{n})$, we bound the $\chi^2$-divergence between the global shift and the local shift by $O(k^2 \delta^4/n)$. By choosing the constant $\delta$ to be sufficiently small, this divergence becomes of lower order than the estimator's precision, making the valid local corruption statistically indistinguishable from the global shift. 
    \item Since the estimator is forced to track the global shift to preserve its accuracy, and it cannot distinguish the valid local corruption from the global shift, the adversary's valid $\eta$-corruption forces the estimator to have a large displacement of $\Omega(\eta \delta) = \Omega(\eta)$. Taking the supremum over the prior establishes the $\Omega(\eta)$ empirical sensitivity lower bound. 
\end{enumerate}

The complete proof can be found in \cref{app:mean-low}.
\end{proofsketch}

\begin{theorem}[Mean obstruction: high-corruption regime]
\label{thm:mean-high}
Let $f \colon (\mathbb R^d)^n\to\mathbb R^d$ be measurable and $\left(C_{\mathsf{MSE}}, \frac{d}{n}\right)$-uniformly MSE-accurate for the class of Gaussians $\{\mathcal{N}(\mu, \mathbb{I})\}_{\mu \in \R^d}$. There exist constants $C_{\mathsf{high}}=C_{\mathsf{high}}(C_{\mathsf{MSE}})>0$ and $\useconstant{const:meanhighLower}=\useconstant{const:meanhighLower}(C_{\mathsf{MSE}})>0$ such that, for all sufficiently large $n$ and every $\eta \in (0,\frac{1}{10})$\footnote{The $\frac{1}{10}$ is an artifact of the technical aspects of the proof, and could potentially be boosted to a larger constant.} such that $\lfloor \eta n \rfloor \ge C_{\mathsf{high}}\log n$, we have
\[ \sup_{\mu\in\mathbb R^d} \ES_{\eta,2}(f;\mu) \ge \useconstant{const:meanhighLower} \eta. \]
\end{theorem}

\begin{proofsketch}
We establish the high-corruption lower bound through a coupling argument. As previously stated, we may assume that $\mu' \in [0,1]$ and thus it suffices to analyze a clipped scalar estimator $h \colon \mathbb{R}^n \to [0,1]$. 
\begin{enumerate}
    \item We exploit the estimator's required accuracy. Because $h$ maintains $O(1/n)$ integrated risk over $[0,1]$, its expected output must track the true parameter at the boundaries of the interval. We show that evaluating the estimator across the interval forces its mean response to increase by $\Omega(\eta)$ on average when the underlying distribution is shifted by $\eta$. 
    \item We construct our adversary using coordinate-wise maximal coupling as in \cref{thm:maximal-coupling}. For any parameter $\mu'$, we couple samples from $\mathcal{N}(\mu', 1)^{\otimes n}$ and $\mathcal{N}(\mu' + \eta, 1)^{\otimes n}$. The expected number of different coordinates will be governed by the TV distance between their Gaussian marginals, which is strictly less than $\eta$. Therefore, for $k = \Omega(\log n)$, a Chernoff bound (\cref{thm:chernoff}) can guarantee that the samples differ by at most $k$ points with high probability. Averaging the displacement across the parameter interval establishes the $\Omega(\eta)$ empirical sensitivity lower bound. 
\end{enumerate}

The complete proof can be found in \cref{app:mean-high}. 
\end{proofsketch}

\subsection{Variance obstruction}
The variance obstruction is rooted in a fundamental statistical trade-off: any estimator achieving uniform MSE accuracy must exhibit an output variance of order $\Omega(d/n)$ at some parameter point. An adversary can thus exploit this inherent variability, leveraging the estimator's own structural fluctuations to induce empirical sensitivity. This is made precise by the following theorem. 

\begin{theorem}[Variance obstruction]
\label{thm:variance-main}
Let $f \colon (\mathbb{R}^d)^n \to \mathbb{R}^d$ be measurable and $\left(C_{\mathsf{MSE}}, \frac{d}{n}\right)$-uniformly MSE-accurate for the class of Gaussians $\{\mathcal{N}(\mu, \mathbb{I})\}_{\mu \in \R^d}$. There is a universal constant $\useconstant{const:varLower} > 0$ such that for every $\eta \in (0,1)$, we have
\[ \sup_{\mu \in \mathbb{R}^d} \ES_{\eta, 2}(f; \mu) \ge \useconstant{const:varLower} \sqrt{\frac{\eta d}{n}} .\]

\end{theorem}

\begin{proofsketch}
The proof of the variance obstruction proceeds in three stages: establishing a scalar variance bound, lifting it to high dimensions, and converting variance into empirical sensitivity.
\begin{enumerate} 
\item First, we show that any one-dimensional uniformly MSE-accurate estimator must have output variance $\Omega(1/n)$ at some parameter. Assume that the one-dimensional estimator is $g$ with parameter $\mu'$. Let $a(\mu') \coloneqq \mathbb{E}_X[g(X)]$. Uniform MSE accuracy thus implies $|a(\mu') - \mu'| \lesssim n^{-1/2}$. We consider an interval of large length which we assume the parameter $\mu'$ lies in, and we partition this interval such that each part is roughly of size $n^{-1/2}$, i.e., $\mu'_j = jn^{-1/2}$. Then $a(\mu')$ must move by a constant fraction of the interval length across the grid. For adjacent grid points, the two distributions $\mathcal{N}(\mu'_j, 1)^{\otimes n}$ and $\mathcal{N}(\mu'_{j+1}, 1)^{\otimes n}$ have bounded $\chi^2$-divergence. One can then use Hammersley-Chapman-Robbins inequality (\cref{thm:hcr}) to get
\[ \Var_{\mu'_j}(g) \gtrsim (a(\mu'_{j+1}) - a(\mu'_j))^2. \]
Summing over $j$ and applying Cauchy-Schwarz forces the average scalar variance over the grid to be $\Omega(1/n)$. 

\item Second, we lift this scalar variance lower bound to $d$ dimensions by applying this one-dimensional finite-difference argument along each coordinate direction, with all other mean coordinates fixed on a common finite grid. Doing this give us some parameter $\mu^{\ast}$ such that $\Var_{\mu^{\ast}}(f(X)) \gtrsim \frac{d}{n}$. 

\item Finally, we convert this clean-sample variance bound into empirical sensitivity by block resampling. Work on parameter $\mu^{\ast}$. Partition the sample into blocks of size $k \coloneqq \lfloor \eta n \rfloor$. Let $X^{(i)}$ be obtained from $X$ by replacing the $i$th block by an independent fresh block from the same distribution $\mathcal{N}(\mu^{\ast}, \mathbb{I}_d)$. Each $X^{(i)}$ differs from $X$ in at most $k$ samples, and so pointwise 
\[ S_{\eta}^f(X) \ge \| f(X^{(i)}) - f(X) \|. \]
Using Vector Efron-Stein Inequality (\cref{thm:efron-stein}) would gives us $\ES_{\eta, 2}(f; \mu^{\ast}) \gtrsim \sqrt{\frac{\eta d}{n}}$, as desired. 
\end{enumerate} 

The complete proof can be found in \cref{app:variance-main}. 
\end{proofsketch}

\section{Related Bernoulli Results}
\label{sec:bernoulli}
\newconstant{const:bernoulli}
The same empirical-sensitivity question can be asked for Bernoulli mean estimation. We record a simple one-dimensional analogue of the mean obstruction. 

Indeed, for $x \in \{ 0, 1 \}^n$, we let $|x| \coloneqq \sum_{i = 1}^n x_i$ denote its Hamming weight. For a scalar estimator $f \colon \{ 0, 1 \}^n \to \mathbb{R}$, define
\[ S_\eta^f(x) \coloneqq \sup_{y \in \{ 0, 1 \}^n \colon d_H(x,y) \le \lfloor \eta n \rfloor} |f(y) - f(x)|.\]
\begin{theorem}
\label{thm:bernoulli-main}
    Let $f \colon \{ 0, 1 \}^n \to \mathbb{R}$ be measurable and suppose that
    \[ \sup_{p \in [0,1]} \mathbb{E}_{X \sim \mathsf{Bern}(p)^{\otimes n}} |f(X) - p| \le \frac{C_{\mathsf{Bern}}}{\sqrt{n}}. \]
    Then for every $\eta \in (0,1)$ such that $k = \lfloor \eta n \rfloor \ge 1$, and for all sufficiently large $n$ depending on $C_{\mathsf{Bern}}$, we have 
    \[ \sup_{p \in [0,1]} \mathbb{E}_{X \sim \mathsf{Bern}(p)^{\otimes n}} [S_\eta^f(X)] \ge \useconstant{const:bernoulli}\eta.\]
\end{theorem}

\begin{proofsketch}
We will show the empirical sensitivity lower bound through a layer-telescoping style argument on the Boolean hypercube. As shown previously, clipping preserves accuracy and only decreases sensitivity, so it suffices to analyze a bounded scalar estimator $f \colon \{0,1\}^n \to [0,1]$.
\begin{enumerate}
\item We exploit the estimator's required accuracy. By decomposing the uniform prior on the Bernoulli parameter into discrete Hamming-weight layers, we can evaluate the estimator's expected output conditionally per layer. Because $f$ maintains $O(1/\sqrt{n})$ expected risk, binomial anti-concentration bounds guarantee that its expected output at widely separated boundary layers must closely track the true proportion of ones. This forces an $\Omega(\eta)$ gap in the estimator's mean between these boundaries.
\item We then distribute this global gap to local neighborhoods using a transport coupling. We couple elements between layers separated by a distance $\ell = \Theta(\eta n)$ by uniformly flipping exactly $\ell$ zero coordinates to ones, ensuring the coupled instances are strictly within the allowed Hamming radius $\lfloor \eta n \rfloor$. This then shows that the required boundary gap must be bridged by significant intermediate jumps. Averaging this forced variation across the discrete layers establishes the strict $\Omega(\eta)$ lower bound on expected empirical sensitivity.
\end{enumerate}

The complete proof can be found in \cref{app:bernoulli-main}. 
\end{proofsketch}

\bibliographystyle{alpha}
\bibliography{biblio}

\appendix 
\section{Standard Facts}
\label{app:standard-facts}

We collect the standard facts used throughout the paper.

\begin{definition}
Let $P,Q$ be probability measures on a common measurable space. If $Q \ll P$, we define the $\chi^2$-divergence by $\chi^2(Q\|P) := \mathbb{E}_P\left[\left(\frac{dQ}{dP}-1\right)^2\right]$. The total variation distance is defined by $\operatorname{TV}(P,Q) := \sup_A |P(A)-Q(A)|$. 
\end{definition}

\begin{definition}
\label{def:scalar-gaussian-notation}
Let $g: \R^n\to\R$ be a measurable scalar estimator. For each of the following, $\mu'\in\R$ will be fixed and we assume the data is drawn as $X=(X_1,\dots,X_n) \sim \mathcal N(\mu',1)^{\otimes n}$.
We use the following notation throughout:
\begin{itemize}
    \item $a(\mu') := \mathbb{E}_X[g(X)]$ is the mean of the estimator. 
    \item $b(\mu') := a(\mu') - \mu'$ is the bias of the estimator. 
\end{itemize}
\end{definition}

\begin{theorem}[Efron--Stein inequality]
\label{thm:efron-stein}
Suppose that $X_1,\dots,X_n,X_1',\dots,X_n'$ are independent, with $X_i'$ distributed identically to $X_i$ for every $i$. Let $X=(X_1,\dots,X_n)$ and 
\[ X^{(i)}=(X_1,\dots,X_{i-1},X_i',X_{i+1},\dots,X_n), \quad i = 1, \dots, n. \]  Then, for every scalar function $f$ with finite variance, \[ \Var(f(X)) \le \frac12\sum_{i=1}^n \mathbb{E}\left[(f(X)-f(X^{(i)}))^2\right]. \]  More generally, if $f$ is vector-valued and $\mathbb{E}\|f(X)\|_2^2<\infty$, then 
\[ \mathbb{E}\left[\|f(X)-\mathbb{E} f(X)\|_2^2\right] \le \frac12\sum_{i=1}^n \mathbb{E}\left[\|f(X)-f(X^{(i)})\|_2^2\right]. \] 
\end{theorem}

\begin{theorem}[Hammersley--Chapman--Robbins inequality]
\label{thm:hcr}
Let $P,Q$ be probability measures with $Q\ll P$ and $\chi^2(Q\|P)<\infty$. Let $T$ be a statistic such that $T\in L^2(P)$ and $T\in L^1(Q)$. Then \[ \Var_P(T) \ge \frac{\left(\mathbb{E}_Q T-\mathbb{E}_P T\right)^2}{\chi^2(Q\|P)}. \] 
\end{theorem}

\begin{theorem}[Cramér--Rao inequality]
\label{thm:cramer-rao}
Let $\{P_{\mu'}\}_{\mu'\in\Theta}$ be a one-dimensional parametric family with densities $p_{\mu'}$. Assume the family and a given statistic $T(X)$ satisfy the following regularity conditions:
\begin{itemize}
    \item The support $\{x : p_{\mu'}(x) > 0\}$ is independent of the parameter $\mu'$.
    \item The score function $s_{\mu'}(X) \coloneqq \frac{\partial}{\partial \mu'} \log p_{\mu'}(X)$ exists almost everywhere.
    \item The Fisher information $I(\mu') \coloneqq \mathbb{E}_{\mu'}[s_{\mu'}(X)^2]$ satisfies $0 < I(\mu') < \infty$.
    \item The statistic $T$ has finite variance, its mean response $m_T(\mu') \coloneqq \mathbb{E}_{\mu'}[T]$ is differentiable, and it satisfies the score identity $m_T'(\mu') = \mathbb{E}_{\mu'}[T(X)s_{\mu'}(X)]$.
\end{itemize}
Then, the variance of the statistic is lower bounded by
\[ \Var_{\mu'}(T) \ge \frac{(m_T'(\mu'))^2}{I(\mu')}. \]
In particular, if $X=(X_1,\dots,X_n)\sim \mathcal N(\mu',1)^{\otimes n}$, then $I(\mu')=n$, and hence $\Var_{\mu'}(T) \ge \frac{(m_T'(\mu'))^2}{n}$.
\end{theorem}

\begin{theorem}
\label{thm:maximal-coupling}
For any two probability measures $P,Q$ on a common measurable space, there exists a coupling $(X,Y)$ with $X\sim P$ and $Y\sim Q$ such that $\mathbb{P}(X\ne Y)=\operatorname{TV}(P,Q)$. Consequently, if $(X_i,Y_i)_{i=1}^n$ are independent maximal couplings of $P$ and $Q$, then 
\[ \#\{i\in[n]:X_i\ne Y_i\} \sim \Bin(n,\operatorname{TV}(P,Q)). \] 
\end{theorem}

\begin{theorem}[Chernoff bounds]
\label{thm:chernoff}
Let $Z\sim\Bin(n,p)$ and write $\lambda=np$. Then, for every $\delta>0$, $\mathbb{P}(Z\ge (1+\delta)\lambda) \le \left(\frac{e^\delta}{(1+\delta)^{1+\delta}}\right)^\lambda$. Equivalently, for every $t>\lambda$, $\mathbb{P}(Z\ge t) \le \left(\frac{e\lambda}{t}\right)^t$. In particular, for every fixed $\rho\in(0,1)$, if $\lambda\le \rho k$, then $\mathbb{P}(Z\ge k) \le \exp(-c_\rho k)$, where $c_\rho>0$ is a constant depending only on $\rho$.
\end{theorem}

\begin{lemma}
\label{lem:binomial-point-mass}
There exists a universal constant $c>0$ such that for every integer $n\ge 1$ and every $r\in\{0,1,\dots,n\}$, $\mathbb{P}\left( \Bin\left(n,\frac{r}{n}\right)=r \right) \ge \frac{c}{\sqrt{r+1}}$.
\end{lemma}

\section{Empirical Sensitivity of Median}
\label{sec:emp-median}
Throughout this section, we assume for simplicity that $n$ is odd. Write $n = 2m - 1$ and for $x = (x_1, \dots, x_n) \in \mathbb{R}^n$, let $x_{(1)} \le \dots \le x_{(n)}$ denote the order statistics. Therefore, $\operatorname{med}(x) \coloneqq x_{(m)}$ and
\[ S_\eta^{\mathrm{med}}(x) \coloneqq \sup_{y \in \mathbb{R}^n: d_H(x,y) \le k} |\mathrm{med}(x) - \mathrm{med}(y)|\]
where $k \coloneqq \lfloor \eta n \rfloor$. 

\begin{lemma}\label{lem:median-bound}
Let $0 \le k \le m - 1$. Then, for every $x \in \mathbb{R}^n$, 
\[ S_\eta^{\operatorname{med}}(x) \le \max \{ x_{(m+k)} - x_{(m)}, x_{(m)} - x_{(m-k)} \}.\]
\end{lemma}
\begin{proof}
Fix $y\in\mathbb R^n$ with $d_H(x,y)\le k$. Thus $y$ is obtained from $x$ by changing at most $k$ coordinates. We claim that $\operatorname{med}(y)\in [x_{(m-k)},x_{(m+k)}]$. Indeed, there are $m+k$ entries of $x$ that are at most $x_{(m+k)}$. Since at most $k$ coordinates are changed, at least $m$ of these entries remain in $y$. Hence at least $m$ entries of $y$ are at most $x_{(m+k)}$, which implies $\operatorname{med}(y)\le x_{(m+k)}$. Similarly, there are $m+k$ entries of $x$ that are at least $x_{(m-k)}$. After changing at most $k$ coordinates, at least $m$ of these entries remain in $y$. Hence at least $m$ entries of $y$ are at least $x_{(m-k)}$, which implies $\operatorname{med}(y)\ge x_{(m-k)}$. Therefore $\operatorname{med}(y)\in [x_{(m-k)},x_{(m+k)}]$. Since $\operatorname{med}(x)=x_{(m)}$, we obtain $|\operatorname{med}(y)-\operatorname{med}(x)| \le \max\left\{ x_{(m+k)}-x_{(m)}, x_{(m)}-x_{(m-k)} \right\}$. Taking the supremum over all $y$ with $d_H(x,y)\le k$ proves the claim.
\end{proof}
We will also need some results from uniform spacings, see, e.g., ~\cite{Pyke1965Spacings}.
\begin{theorem}
\label{thm:uniform-spacings-pyke}
Let $U_1,\dots,U_n \stackrel{\text{iid}}{\sim} \mathsf{Unif}[0,1]$, and let $0=U_{(0)}\le U_{(1)}\le \cdots \le U_{(n)}\le U_{(n+1)}=1$ denote the order statistics with endpoints added. Define the spacings $D_i \coloneqq U_{(i)}-U_{(i-1)}$ for $i=1,\dots,n+1$. Then $(D_1,\dots,D_{n+1})\sim \operatorname{Dirichlet}(1,\dots,1)$. Equivalently, if $E_1,\dots,E_{n+1} \stackrel{\text{iid}}{\sim} \mathsf{Exp}(1)$ and $T\coloneqq\sum_{i=1}^{n+1}E_i$, then $(D_1,\dots,D_{n+1}) \stackrel{d}{=} \left( \frac{E_1}{T},\dots,\frac{E_{n+1}}{T} \right)$.
\end{theorem}

We can now state the Gaussian empirical sensitivity of the median. 
\begin{theorem} \label{thm:median-emp-sens}
Let $n = 2m - 1$ be odd, and let $X_1, \dots, X_n \stackrel{\text{iid}}{\sim} \mathcal{N}(\mu, 1)$. Fix $\eta \in (0,1/3)$ and set $k = \lfloor \eta n \rfloor$. There exists universal constants $c, C > 0$ such that, for every $\beta \in (0,1/2)$, with probability at least $1 - \beta - Ce^{-cn}$, we have 
\[ S_\eta^{\mathrm{med}}(X) \le C \left( \eta + \frac{\log(1/\beta)}{n} \right).  \]
\end{theorem}
\begin{proof}
By translation invariance, it suffices to prove the result for $\mu = 0$. The case $k = 0$ or $\eta = 0$ directly gives us $S_\eta^{\mathrm{med}}(X) = 0$. So we will assume $k \ge 1$. Let $\Phi$ denote the standard Gaussian distribution function and define $U_i \coloneqq \Phi(X_i)$. Then $U_1, \dots, U_n \stackrel{\text{iid}}{\sim} \mathsf{Unif}[0,1]$. Since $\Phi$ is strictly increasing, $X_{(j)} = \Phi^{-1}(U_{(j)})$ for every $j \in [n]$. We first localize the relevant order statistics away from the tails. Since $\eta < 1/3$ and $k=\lfloor \eta n\rfloor$, we have $m - k \ge \frac{n}{6}$ and $m + k \le \frac{5n}{6} + 1$. 

Let $N_- \coloneqq \sum_{i=1}^n \mathbf 1\{U_i\le 1/8\}$ and $N_+ \coloneqq \sum_{i=1}^n \mathbf 1\{U_i\le 7/8\}$. Then $N_- \sim \operatorname{Bin}(n,1/8)$ and $N_+ \sim \operatorname{Bin}(n,7/8)$. By the binomial Chernoff bound in \cref{thm:chernoff}, there exist universal constants $c,C>0$ such that $\mathbb P\left(N_- \ge \frac n6\right)\le Ce^{-cn}$ and $\mathbb P\left(N_+ \le \frac{5n}{6}\right)\le Ce^{-cn}$. Therefore, with probability at least $1-Ce^{-cn}$, $U_{(m-k)}\ge \frac18$ and $U_{(m+k)}\le \frac78$. On this event, all the order statistics $U_{(m-k)}, U_{(m)}, U_{(m+k)}$ lie in $[1/8,7/8]$.

The inverse Gaussian CDF is Lipschitz on $[1/8,7/8]$. Indeed, $\frac{d}{du}\Phi^{-1}(u) = \frac{1}{\varphi(\Phi^{-1}(u))}$, where $\varphi$ is the standard Gaussian density, and the right-hand side is bounded on the compact interval $[1/8,7/8]$. Hence there exists a universal constant $L<\infty$ such that, for all $u,v\in[1/8,7/8]$, $|\Phi^{-1}(u)-\Phi^{-1}(v)|\le L|u-v|$. Consequently, on the localization event, $X_{(m+k)}-X_{(m)} \le L\bigl(U_{(m+k)}-U_{(m)}\bigr)$ and $X_{(m)}-X_{(m-k)} \le L\bigl(U_{(m)}-U_{(m-k)}\bigr)$.

It remains to control the two uniform order-statistic spacings. By \cref{thm:uniform-spacings-pyke}, each $k$-spacing, such as $U_{(m+k)}-U_{(m)}$ and $U_{(m)}-U_{(m-k)}$, has the same marginal distribution as $A/T$, where $A\sim \operatorname{Gamma}(k,1)$ is the sum of $k$ iid exponential random variables, and $T\sim \operatorname{Gamma}(n+1,1)$ is the sum of all $n+1$ such variables. We now prove a high-probability bound for $A/T$. Let $t\ge 0$. By Chernoff's inequality applied to the moment generating function of $A$, we have 
\[ \mathbb P(A\ge 2(k+t)) \le e^{-(k+t)} \mathbb E e^{A/2} = e^{-(k+t)}2^k = \exp\{-t+k(\log 2-1)\} \le e^{-t}. \] 
Also, another Chernoff bound gives $\mathbb P(T\le n/2)\le e^{-cn}$ for a universal constant $c>0$. Hence, with probability at least $1-e^{-t}-e^{-cn}$, $\frac{A}{T} \le  4\frac{k+t}{n}$. Applying this bound to the two spacings and taking a union bound gives $\mathbb P\left( \max\left\{ U_{(m+k)}-U_{(m)},\, U_{(m)}-U_{(m-k)} \right\} > 4\frac{k+t}{n} \right) \le 2e^{-t}+2e^{-cn}$. Taking $t=\log(4/\beta)$, we obtain $\max\left\{ U_{(m+k)}-U_{(m)},\, U_{(m)}-U_{(m-k)} \right\} \le C\frac{k+\log(1/\beta)}{n}$ with probability at least $1-\beta/2-Ce^{-cn}$.

Using union bound, we get that with probability at least $1-\beta-Ce^{-cn}$, 
\[ \max\left\{ X_{(m+k)}-X_{(m)},\, X_{(m)}-X_{(m-k)} \right\} \le C\frac{k+\log(1/\beta)}{n} \le C\left( \eta+\frac{\log(1/\beta)}{n} \right). \] 
Finally, by \cref{lem:median-bound}, $S_\eta^{\operatorname{med}}(X) \le \max\left\{ X_{(m+k)}-X_{(m)},\, X_{(m)}-X_{(m-k)} \right\}$. Therefore, with probability at least $1-\beta-Ce^{-cn}$, we conclude $S_\eta^{\operatorname{med}}(X) \le C\left( \eta+\frac{\log(1/\beta)}{n} \right)$, which proves the theorem.
\end{proof}

\section{Resampling Adversary Upper Bound}
\label{sec:resampling-adversary}
In this section, we record a simple observation showing that the linear-in-$\eta$ term in our corruption lower bound should not be expected under the resampling adversary.

Recall the resampling adversary:
\begin{definition}[Resampling Adversary] Fix $\eta \in (0,1)$ and let $k \coloneqq \lfloor \eta n \rfloor$. Given $X = (X_1, \dots, X_n)$ with $X \sim \mathcal{N}(\mu, \mathbb{I}_d)^{\otimes n}$. Let $I \subset [n]$ be a uniformly random subset of size $k$, independent of $X$. Let $(X_i')_{i \in I}$ be independent fresh samples from $\mathcal{N}(\mu, \mathbb{I}_d)$, also independent of $(X, I)$. The resampled dataset $X^{\mathrm{res}}$ is defined by
\[ X_i^{\mathrm{res}} \coloneqq \begin{cases} X_i', &\text{if } i \in I, \\ X_i, &\text{if } i \notin I .\end{cases} \]
\end{definition}

The following proposition quantifies how much the empirical mean shifts under a resampling adversary.
\begin{proposition}
    Let $\overline{X} \coloneqq \frac{1}{n} \sum_{i = 1}^n X_i$ and $\overline{X}^{\mathrm{res}} \coloneqq \frac{1}{n} \sum_{i =1}^n X_i^{\mathrm{res}}$. Then 
    \[ \left( \mathbb{E}_{X, I, X'} \| \overline{X}^{\mathrm{res}} - \overline{X} \|_2^2 \right)^{1/2} \asymp \sqrt{\frac{\eta d}{n}}.\]
\end{proposition}
    Consequently, under the resampling adversary, the empirical mean has empirical stability of order $\sqrt{\eta d/n}$, and does not exhibit an $\eta$ term. 

\begin{proof}
By definition, we have
\[ \overline{X}^{\mathrm{res}} - \overline{X} = \frac{1}{n} \sum_{i \in I} (X_i' - X_i). \]
Conditional on $I$, the random vectors $(X_i' - X_i)_{i \in I}$ are independent, and each has law $\mathcal{N}(0,2 \mathbb{I}_d)$. Therefore, 
\[ (\overline{X}^{\mathrm{res}} - \overline{X}) \mid I \sim \mathcal{N} \left( 0, \frac{2k}{n^2} \mathbb{I}_d \right) \]
Hence, 
\[ \mathbb{E}_{X, X' \mid I} \| \overline{X}^{\mathrm{res}} - \overline{X} \|_2^2 = \operatorname{tr} \left( \frac{2k}{n^2} \mathbb{I}_d \right) = \frac{2kd}{n^2}. \]
Since RHS is deterministic in $I$, averaging over $I$ gives the same value. Therefore, taking square root yields
\[ \left( \mathbb{E}_{X, I, X'}\| \overline{X}^{\mathrm{res}} - \overline{X} \|_2^2  \right)^{1/2} = \sqrt{\frac{2kd}{n^2}} \asymp \sqrt{\frac{\eta d}{n}}.  \]
\end{proof}

\section{Lemmas} 
In this section, we isolate the main lemma that allow us to prove the high-dimensional theorem for the mean obstruction by reducing the problem to a scalar setting. Ultimately, we claim that it suffices to prove the lower bound for a bounded, one-dimensional estimator $h \colon \mathbb{R}^n \to [0,1]$ that inherits $O(1/n)$ integrated risk. 

This restriction to compact interval $[0,1]$ is rigorously justified by the following reduction: first, we restrict the parameter space of the projected 1D problem to the interval $\mu' \in [0,1]$. Because the original high-dimensional estimator $f$ is uniformly MSE-accurate over all $\mathbb{R}^d$, its 1D projection must maintain low average risk over this specified bounded sub-interval under a Gaussian prior. Because the true parameter $\mu$ is now restricted to $[0,1]$, any output of the estimator outside this interval is strictly suboptimal, and thus we may clip the estimator's output to $[0,1]$ which strictly decreases or maintains both its estimation and its empirical sensitivity. Consequently, we may establish an empirical sensitivity lower bound for the clipped version of the estimator $h$ to guarantee a valid lower bound for the original unrestricted high-dimensional estimator $f$. 

Let $V = \mathsf{span}(u)$. If $\lambda \in u^{\perp}$ and $Z \sim \mathcal{N}(0,\mathbb{I}_d)$, then $\lambda + (\mathbb{I}_d - uu^\top) Z$ is drawn from a Gaussian distribution on the affine subspace $\lambda + V^{\perp}$ with mean $\lambda$. 

\begin{lemma}\label{lem:projection-conditioning}
Let $f \colon (\R^d)^n\to\R^d$ be measurable. Fix $u\in\mathbb S^{d-1}$ and $\lambda\in u^\perp$. For $t=(t_1,\dots,t_n)\in\R^n$, define
\[ g_\lambda(t) \coloneqq \E_Z\left[ \left\langle u, f(t_1u+V_1,\dots,t_nu+V_n) \right\rangle \right], \qquad V_i \coloneqq \lambda+(\mathbb{I}_d-uu^\top)Z_i, \]
where $Z_1,\dots,Z_n\stackrel{\mathrm{iid}}{\sim}\mathcal N(0,\mathbb{I}_d)$. Then for every $\eta\in(0,1)$, every $q\in\{1,2\}$, and every $\mu'\in\R$,
\[ \left( \E_{T\sim\mathcal N(\mu',1)^{\otimes n}} \bigl[(S_\eta^{g_\lambda}(T))^q\bigr] \right)^{1/q} \le \ES_{\eta,q}(f;\mu' u+\lambda). \]

Moreover,
\[ \E_{T\sim\mathcal N(\mu',1)^{\otimes n}} [(g_\lambda(T)-\mu')^2] \le \E_{X\sim\mathcal N(\mu' u+\lambda,\mathbb{I}_d)^{\otimes n}} \bigl[\langle u,f(X)-(\mu' u+\lambda)\rangle^2\bigr]. \]
\end{lemma}

\begin{proof}
Set $k \coloneqq \lfloor\eta n\rfloor$. Fix $t,t'\in\R^n$ with $\Hamm(t,t')\le k$, and use the same orthogonal noise to lift both
scalar samples: $x_i = t_i u + V_i$ and $x_i' = t_i' u + V_i$. Then $\Hamm(x, x') \le k$ and so
\[ \begin{aligned} |g_\lambda(t')-g_\lambda(t)| &= \left|\E_Z\langle u,f(x')-f(x)\rangle\right| \\ &\le \E_Z\norm{f(x')-f(x)} \\ &\le \E_Z S_\eta^f(x). \end{aligned} \]

Taking the supremum over all such $t'$ gives
\[ S_\eta^{g_\lambda}(t) \le \E_Z S_\eta^f(t_1u+V_1,\dots,t_nu+V_n), \]
which implies the sensitivity comparison we want by Jensen's inequality and averaging over $T \sim \mathcal{N}(\mu', 1)^{\otimes n}$. 

For the risk comparison, Jensen gives
\[ \begin{aligned} \E_T[(g_\lambda(T)-\mu')^2] &= \E_T\left[ \left( \E_Z[\langle u,f(T_1u+V_1,\dots,T_nu+V_n)\rangle-\mu'] \right)^2 \right] \\ &\le \E_{T,Z} \bigl[ (\langle u,f(T_1u+V_1,\dots,T_nu+V_n)\rangle-\mu')^2 \bigr]. \end{aligned} \]

For each $i$,
\[ T_i u+V_i = T_i u+\lambda+(\mathbb{I}_d-uu^\top)Z_i \sim \mathcal N(\mu' u+\lambda,\mathbb{I}_d). \]

Also $\langle u,\mu' u+\lambda\rangle=\mu'$, since $\lambda\in u^\perp$. Therefore the last term above is just
\[ \E_{X\sim\mathcal N(\mu' u+\lambda,\mathbb{I}_d)^{\otimes n}} \bigl[\langle u,f(X)-(\mu' u+\lambda)\rangle^2\bigr], \]
which proves the claim.
\end{proof}

\begin{lemma}
\label{lem:good-section}
Let $f \colon (\R^d)^n\to\R^d$ be measurable and $\left(C_{\mathsf{MSE}}, \frac{d}{n}\right)$-uniformly MSE-accurate. Fix $\rho>0$. Then there exist $u\in\mathbb S^{d-1}$ and $\lambda_*\in u^\perp$ such that the scalar estimator $g_{\lambda_*}$ from \cref{lem:projection-conditioning} satisfies
\[ \E_{\mu'\sim\mathcal N(0,\rho^2)} \E_{T\sim\mathcal N(\mu',1)^{\otimes n}} [(g_{\lambda_*}(T)-\mu')^2] \le \frac{C_{\mathsf{MSE}}}{n}. \]
\end{lemma}

\begin{proof}
Let $M\sim\mathcal N(0,\rho^2\mathbb{I}_d)$ and let
$U\sim\Unif(\mathbb S^{d-1})$, independently of all other randomness.
For every fixed $z\in\R^d$,
\[ \E_U[\langle U,z\rangle^2]=\frac{\norm z^2}{d}. \]

Hence, by uniform MSE accuracy,
\[ \begin{aligned} &\E_U\E_M\E_{X\sim\mathcal N(M,\mathbb{I}_d)^{\otimes n}} [\langle U,f(X)-M\rangle^2] \\ &\qquad= \frac1d\E_M\E_{X\sim\mathcal N(M,\mathbb{I}_d)^{\otimes n}} \norm{f(X)-M}^2 \le \frac{C_{\mathsf{MSE}}}{n}. \end{aligned} \]

Therefore there exists $u\in\mathbb S^{d-1}$ such that
\[ \E_M\E_{X\sim\mathcal N(M,\mathbb{I}_d)^{\otimes n}} [\langle u,f(X)-M\rangle^2] \le \frac{C_{\mathsf{MSE}}}{n}. \]

Decompose
\[ M=\mu' u+\lambda,\qquad \mu' \coloneqq \langle M,u\rangle,\qquad \lambda \coloneqq M-\langle M,u\rangle u. \]

Then $\mu'\sim\mathcal N(0,\rho^2)$, $\lambda\in u^\perp$, and
$\mu'$ is independent of $\lambda$. Applying
\cref{lem:projection-conditioning} and averaging over $\lambda$ gives
\[ \begin{aligned} &\E_\lambda\E_{\mu'\sim\mathcal N(0,\rho^2)} \E_{T\sim\mathcal N(\mu',1)^{\otimes n}} [(g_\lambda(T)-\mu')^2] \\ &\qquad\le \E_M\E_{X\sim\mathcal N(M,\mathbb{I}_d)^{\otimes n}} [\langle u,f(X)-M\rangle^2] \le \frac{C_{\mathsf{MSE}}}{n}. \end{aligned} \]

Thus some $\lambda_*\in u^\perp$ satisfies the desired bound.
\end{proof}

\begin{lemma}
\label{lem:clipping}
Let $K \subseteq \mathbb{R}$ be a closed interval and let $\mathrm{clip}_K(t)$ denote Euclidean projection onto $K$. For any scalar estimator $g$, define $h \coloneqq \mathrm{clip}_K \circ g$. Then, for every $x \in \mathbb{R}^n$, we have $S_\eta^h(x) \le S_\eta^g(x)$. Moreover, if $\mu' \in K$, then for every $x$, $|h(x) - \mu'| \le |g(x) - \mu'|$. 
\end{lemma}
\begin{proof}
    The projection map $\mathrm{clip}_K$ is $1$-Lipschitz and fixes every point of $K$. Thus, if $\mu \in K$, we have 
    \[ |h(x) - \mu'| = |\mathrm{clip}_K(g(x)) - \mathrm{clip}_K(\mu')| \le |g(x) - \mu'|. \]
    Similarly, for any $x, x'$, we obtain $|h(x') - h(x)| \le |g(x') - g(x)|$. Taking the supremum over all $x'$ with $d_H(x, x') \le \lfloor \eta n \rfloor$ proves the sensitivity bound we want. 
\end{proof}
\begin{lemma} \label{lem:scalar}
    Let $f\colon (\mathbb{R}^d)^n \to \mathbb{R}^d$ be measurable and $\left(C_{\mathsf{MSE}}, \frac{d}{n}\right)$-uniformly MSE-accurate. There exists $u \in \mathbb{S}^{d - 1}, \lambda_{\ast} \in u^{\perp}$, and a scalar estimator $h \colon \mathbb{R}^n \to [0,1]$ such that, for every $\eta \in (0,1)$, every $q \in \{ 1, 2 \}$, and every $\mu' \in [0,1]$, 
    \[ (\mathbb{E}_{X \sim \mathcal{N}(\mu', 1)^{\otimes n}}[S_\eta^h(X)^q])^{1/q} \le \ES_{\eta, q} (f; \mu' u + \lambda_{\ast})\]
    and 
    \[ \int_0^1 \mathbb{E}_{X \sim \mathcal{N}(\mu', 1)^{\otimes n}}[(h(X) - \mu')^2] \ d\mu' \le \frac{C_{\mathsf{int}}}{n}, \]
    where $C_{\mathsf{int}} = C_{\mathsf{int}}(C_{\mathsf{MSE}})$. 
\end{lemma}
\begin{proof}
    Apply \cref{lem:good-section} with $\rho = 1$ to obtain $u$ and $\lambda_{\ast}$. Let $g \coloneqq g_{\lambda_{\ast}}$ be the scalar estimator from \cref{lem:projection-conditioning}, and define $h \coloneqq \textrm{clip}_{[0,1]} \circ g$. By \cref{lem:clipping}, $S_\eta^h(x) \le S_\eta^g(x)$ for every $x$, and for every $\mu' \in [0,1]$, $|h(x) - \mu'| \le |g(x) - \mu'|$. The sensitivity comparison thus follow from \cref{lem:projection-conditioning}. It remains to prove the risk bound. Indeed, let $\varphi$ denote the density of $\mathcal{N}(0,1)$ and set $c_0 \coloneqq \inf_{\mu' \in [0,1]} \varphi(\mu') > 0$. Then, 
    \begin{align*}
      \int_0^1 \mathbb{E}_{X \sim \mathcal{N}(\mu', 1)^{\otimes n}}[(h(X) - \mu')^2] \ d\mu' &\le \int_0^1 \mathbb{E}_{X \sim \mathcal{N}(\mu', 1)^{\otimes n}}[(g(X) - \mu')^2] \ d\mu' \\ 
      &\le c_0^{-1} \mathbb{E}_{\mu' \sim \mathcal{N}(0,1)}\mathbb{E}_{X \sim \mathcal{N}(\mu', 1)^{\otimes n}}[(g(X) - \mu')^2]  \\ 
      &\le \frac{c_0^{-1} C_{\mathsf{MSE}}}{n}.
    \end{align*}
    Therefore, the claim holds with $C_{\mathsf{int}} \coloneqq c_0^{-1} C_{\mathsf{MSE}}$. 
\end{proof}
\section{Proof of~\texorpdfstring{\cref{thm:upper-main}}{}}
\label{app:upper-main}
We will rely on the following estimator that is recently studied in \cite{ChenDMM26} in the context of empirical mean recovery. 
\begin{theorem}[Theorem 5.5 from \cite{ChenDMM26}] \label{thm:CDMM} Suppose $0 \le \eta < \frac{1}{2}$ is bounded away from $\frac{1}{2}$. There exists a (computationally inefficient) $\eta$-robust estimator that, for all $\mu \in \mathbb{R}^d$, given $n$ iid samples from $\mathcal{N}(\mu, \mathbb{I}_d)$ with empirical mean $\overline{x}$, outputs $\hat{x}$ such that $\| \hat{x} - \overline{x} \|_2 \le \alpha$, with probability at least $1 - \beta$, for 
\[ \alpha = O \left( \eta \sqrt{\log (1/\eta) } + \sqrt{\eta \cdot \frac{d + \log (1/\beta)}{n}} \right) \]
\end{theorem}

We will use a slightly strengthened formulation that follows immediately from the proof of \cref{thm:CDMM}. This estimator operates by searching over candidate cleaned datasets $Y \in (\mathbb{R}^d)^n$ that differ from the corrupted input $\widetilde{X}$ in at most $\eta n$ coordinates that satisfy some constraints. We denote the set of feasible candidates in this search as $\mathcal{F}_{\eta, \beta}(\widetilde{X})$. The proof of \cref{thm:CDMM} establishes the following two uniform properties in the setup of \cref{thm:CDMM}: for every $\eta$-corruption $\widetilde{X}$ of $X$, 
\begin{enumerate}
    \item the clean dataset $X$ is feasible, i.e., $X \in \mathcal{F}_{\eta, \beta}(\widetilde{X})$, and 
    \item every feasible candidate $Y \in \mathcal{F}_{\eta, \beta}(\widetilde{X})$ satisfies $\| \overline{Y} - \overline{X} \|_2 \le \alpha$, where $\overline{Y} \coloneqq \frac{1}{n} \sum_{i = 1}^n Y_i$. 
\end{enumerate}
The original estimator in \cref{thm:CDMM} may output $\perp$ if the feasibility system is empty. We will slightly modify this as follow to ensure $f_{\eta, \beta}$ is a measurable map. Given an input dataset $Z = (Z_1, \dots, Z_n)$, if $\mathcal{F}_{\eta, \beta}(Z) = \varnothing$, we set $f_{\eta, \beta}(Z) \coloneqq \overline{Z}$. Otherwise, we select a feasible candidate $Y^{\ast}(Z)$ that minimizes the Hamming distance to the input: 
\[ Y^{\ast}(Z) \in \argmin_{Y \in \mathcal{F}_{\eta, \beta}(Z)} |\{ i \in [n]\colon Y_i \not= Z_i \}|, \]
where ties are broken arbitrarily, and we set $f_{\eta, \beta}(Z) \coloneqq \overline{Y^{\ast}(Z)}$. We note that this tie-breaking procedure preserves the guarantee of \cref{thm:CDMM} because the guarantee there applies to \textit{all} feasible solutions. 

With this modification in mind, we see that on the high probability event established in \cref{thm:CDMM}, if the input is not corrupted, i.e., $Z = X$, then $f_{\eta, \beta}(X) = \overline{X}$. Indeed, to see this, note that $X \in \mathcal{F}_{\eta, \beta}(X)$ and since this candidate has zero disagreements with the input $X$, it uniquely minimizes the Hamming distance, and thus $Y^{\ast}(X) = X$, giving us $f_{\eta, \beta}(X) = \overline{X}$. 

Using this, we can now bound the empirical sensitivity and clean MSE bound. Indeed, assume the high probability event from \cref{thm:CDMM} holds with the bad event happening with probability at most $\frac{\beta}{2}$. We see that 
\begin{align*}
    S_\eta^{f_{\eta,\beta}}(X) &= \sup_{X' \colon d_H(X, X') \le \lfloor \eta n \rfloor} \| f_{\eta, \beta}(X') - f_{\eta, \beta}(X) \|_2 \\ 
    &= \sup_{X' \colon d_H(X, X') \le \lfloor \eta n \rfloor} \| f_{\eta, \beta}(X') - \overline{X} \|_2 \le \alpha = \widetilde{O} \left( \eta + \sqrt{\eta \cdot \frac{d + \log(1/\beta)}{n}} \right), 
\end{align*}
which is what we wanted. Similarly, on the same event one can note that 
\[ \| f_{\eta, \beta}(X) - \mu \|_2^2 = \| \overline{X} - \mu \|^2_2 = O \left( \frac{d + \log(1/\beta)}{n} \right). \]
with probability $1 - \frac{\beta}{2}$. Union bound establish what we want. 

\section{Proof of~\texorpdfstring{\cref{thm:mean-low}}{}}
\label{app:mean-low}
Throughout this section, for a scalar parameter $\mu' \in \mathbb{R}$, we write $P_{\mu'} \coloneqq \mathcal{N}(\mu', 1)^{\otimes n}$. By the reduction established in \cref{lem:scalar}, to prove the high-dimensional theorem, it suffices to prove the following scalar claim: if the induced one-dimensional estimator $h \colon \mathbb{R}^n \to [0,1]$ satisfies the integrated risk bound
\[ \int_0^1 \mathbb{E}_{X \sim P_{\mu'}}[(h(X) - \mu')^2] \, d\mu' \le \frac{C_{\mathsf{int}}}{n}, \]
then, for a corruption budget $1 \le k = \lfloor \eta n \rfloor \le \kappa \sqrt{n}$, its average empirical sensitivity must be large:
\[ \sup_{\mu' \in [0,1]} \mathbb{E}_{X \sim P_{\mu'}}[S_\eta^h(X)] \ge c_{\mathrm{low}} \eta. \]

To see exactly why this scalar claim implies the full high-dimensional theorem, recall that \cref{lem:scalar} guarantees the existence of a ``good'' unit direction $u \in \mathbb{S}^{d-1}$ and an orthogonal shift $\lambda_{\ast} \in u^\perp$. When we evaluate the original high-dimensional estimator $f$ along the one-dimensional affine line parameterized by $\mu' u + \lambda_{\ast}$, its empirical sensitivity strictly upper-bounds the sensitivity of the induced scalar estimator $h$. Together with Jensen's inequality, this gives us 
\[ \ES_{\eta, 2}(f; \mu' u + \lambda_{\ast}) \ge \left( \mathbb{E}_{X \sim P_{\mu'}}[S_\eta^h(X)^2] \right)^{1/2} \ge \mathbb{E}_{X \sim P_{\mu'}}[S_\eta^h(X)]. \]
Taking the supremum over $\mu' \in [0,1]$ on both sides demonstrates that establishing the $c_{\mathrm{low}} \eta$ lower bound for $h$ immediately yields the required lower bound for $f$, completing the reduction.

\subsection{Statistical Indistinguishability of the Local and Global Shift} 
WLOG we may choose $\eta$ such that $\eta = \frac{\lfloor \eta n \rfloor}{n}$, i.e., $\eta n \in \{ 1, \dots, n \}$. For $\delta \ge 0$ and $\mu' \in \mathbb{R}$, we define the local-shift distribution $Q_{\mu', \delta}$ via the following generative process: draw $X \sim P_{\mu'}$, draw a $k$-subset $I \subset [n]$ uniformly at random (independent of $X$), and output $X^{(I, \delta)} \coloneqq X + \delta e_I$, i.e., this means that
\[ X_i^{(I, \delta)} = \begin{cases} X_i + \delta, &\text{if } i \in I,\\ X_i, &\text{if } i \notin I. \end{cases}\]

The key observation is to show that this local adversarial corruption is statistically indistinguishable from a global mean shift $P_{\mu' + \eta\delta}$. 

\begin{lemma} \label{lem:close-dist-corr}
    For every $\mu' \in \mathbb{R}$ and every $\delta \ge 0$, 
    \[ \chi^2(Q_{\mu', \delta} \| P_{\mu' + \eta \delta }) \le \exp \left( \frac{k^2}{n} (e^{\delta^2} - 1 - \delta^2) \right) - 1. \]
\end{lemma}

To prove \cref{lem:close-dist-corr}, we require two elementary auxiliary lemmas. 

\begin{lemma}[Gaussian likelihood-ratio identity] \label{lem:gaussian-lr}
    For any $\theta, \mu , \nu \in \mathbb{R}^n$, if $X \sim \mathcal{N}(\theta, \mathbb{I}_n)$, then 
    \[ \mathbb{E}_{X \sim \mathcal{N}(\theta, \mathbb{I}_n)} \left[ \frac{d\mathcal{N}(\mu, \mathbb{I}_n)}{d\mathcal{N}(\theta, \mathbb{I}_n)}(X) \cdot \frac{d\mathcal{N}(\nu, \mathbb{I}_n)}{d\mathcal{N}(\theta, \mathbb{I}_n)}(X) \right] = \exp(\langle \mu - \theta, \nu - \theta \rangle). \]
\end{lemma}
\begin{proof}
  Let $a\coloneqq \mu-\theta$ and $b\coloneqq \nu-\theta$. If $X\sim \mathcal N(\theta,\mathbb I_n)$, then $X=\theta+G$ with $G\sim\mathcal N(0,\mathbb I_n)$. The likelihood ratio of $\mathcal N(\mu,\mathbb I_n)$ with respect to $\mathcal N(\theta,\mathbb I_n)$ is $\frac{d\mathcal N(\mu,\mathbb I_n)}{d\mathcal N(\theta,\mathbb I_n)}(X) = \exp\left(\langle a,X-\theta\rangle -\frac12\|a\|_2^2\right) = \exp\left(\langle a,G\rangle -\frac12\|a\|_2^2\right)$. Similarly, $\frac{d\mathcal N(\nu,\mathbb I_n)}{d\mathcal N(\theta,\mathbb I_n)}(X) = \exp\left(\langle b,G\rangle -\frac12\|b\|_2^2\right)$. 
  Therefore, the expected product of the likelihood ratios is $$\mathbb E_{X\sim\mathcal N(\theta,\mathbb I_n)} \left[ \frac{d\mathcal N(\mu,\mathbb I_n)}{d\mathcal N(\theta,\mathbb I_n)}(X) \frac{d\mathcal N(\nu,\mathbb I_n)}{d\mathcal N(\theta,\mathbb I_n)}(X) \right] = \exp\left(-\frac12\|a\|_2^2-\frac12\|b\|_2^2\right) \mathbb E_G \exp\left(\langle a+b,G\rangle\right).$$ 
  Using the standard Gaussian moment-generating function $\mathbb E \exp(\langle t,G\rangle) = \exp(\frac12\|t\|_2^2)$ with $t=a+b$, the expectation evaluates to $\exp\left(-\frac12\|a\|_2^2 -\frac12\|b\|_2^2 +\frac12\|a+b\|_2^2\right)$. Expanding the squared norm gives the identity $\frac12\|a+b\|_2^2 -\frac12\|a\|_2^2 -\frac12\|b\|_2^2 = \langle a,b\rangle$. Substituting back $a=\mu-\theta$ and $b=\nu-\theta$ yields the final expected value $\exp(\langle \mu-\theta,\nu-\theta\rangle)$, as claimed.
\end{proof}

\begin{lemma} \label{lem:hypergeom-mgf}
    Let $A$ and $B$ be independent uniformly random $k$-subsets of $[n]$, and let $H = |A \cap B|$. For every $\lambda \ge 0$, 
    \[ \mathbb{E} \exp \left( \lambda \left( H - \frac{k^2}{n} \right) \right) \le \exp \left( \frac{k^2}{n} (e^{\lambda} - 1 - \lambda) \right). \]
\end{lemma}
\begin{proof}
Condition on $A$. For $i \in A$, write $X_i := \mathbf{1} \{ i \in B \}$. Then $H = |A \cap B| = \sum_{i \in A} X_i$. Set $c := e^{\lambda} - 1 \ge 0$. Since $X_i \in \{ 0, 1 \}$, we have 
\[ e^{\lambda H} = \prod_{i \in A} (1 + cX_i) = \sum_{R \subseteq A} c^{|R|} \mathbf{1} \{ R \subseteq B \}. \]
If $|R| = s$, then 
\[ \mathbb{P}(R \subseteq B) = \frac{\binom{n - s}{k - s}}{\binom{n}{k}} \le \left( \frac{k}{n} \right)^s. \]
Consequently, we have
\[ \mathbb{E}[e^{\lambda H} \mid A] \le \sum_{s = 0}^k \binom{k}{s} c^s \left( \frac{k}{n} \right)^s = \left( 1 + \frac{k}{n} (e^{\lambda} - 1) \right)^k. \]
Since RHS does not depend on $A$, then the same bound hold for $\mathbb{E}[e^{\lambda H}]$ by averaging over all possible $A$. 
Multiplying by $\exp(-\lambda k^2/n)$ and using $\log(1 + x) \le x$ for $x \ge 0$, we thus obtain 
\begin{align*}
    \mathbb{E} \exp \left( \lambda \left( H - \frac{k^2}{n} \right) \right) &\le \exp \left( - \lambda \frac{k^2}{n} + k \log \left( 1 + \frac{k}{n} (e^{\lambda} - 1) \right) \right) \\ 
    &\le \exp \left( \frac{k^2}{n} (e^{\lambda} - 1 - \lambda) \right), 
\end{align*}

\end{proof}

\begin{proof}[Proof of \cref{lem:close-dist-corr}]
    By translation invariance, it suffices to set $\mu' = 0$. Let $P \coloneqq P_{\eta \delta} = \mathcal{N}(\eta \delta, \mathbb{I}_n)$. The local-shift distribution $Q_{0, \delta}$ is a mixture over the subset $I$, so its likelihood ratio with respect to $P$ is
    \[ \frac{dQ_{0, \delta}}{dP}(X) = \mathbb{E}_I \left[ \frac{d\mathcal{N}(\delta e_I, \mathbb{I}_n)}{dP}(X) \right]. \]
    Using \cref{lem:gaussian-lr} to compute the $\chi^2$ divergence, we expand the square of the mixture:
    \begin{align*}
        \chi^2(Q_{0, \delta} \| P) &= \mathbb{E}_{X \sim P} \left[ \left(\frac{dQ_{0, \delta}}{dP}(X)\right)^2 \right] - 1 \\
        &= \mathbb{E}_{I, J} \left[ \exp(\langle \delta e_I - \eta \delta, \delta e_J - \eta \delta \rangle) \right] - 1.
    \end{align*}
    We evaluate the inner product: $\langle \delta e_I - \eta \delta, \delta e_J - \eta \delta \rangle = \delta^2 |I \cap J| - 2\eta \delta^2 k + n \eta^2 \delta^2$. Since $\eta = k/n$, this simplifies to $\delta^2(|I \cap J| - k^2/n)$. Letting $H = |I \cap J|$ and applying \cref{lem:hypergeom-mgf} with $\lambda = \delta^2$ yields:
    \[ \chi^2(Q_{0, \delta} \| P) = \mathbb{E}_{I, J} \left[ \exp\left( \delta^2 \left(H - \frac{k^2}{n}\right) \right) \right] - 1 \le \exp \left( \frac{k^2}{n} (e^{\delta^2} - 1 - \delta^2) \right) - 1, \]
    concluding the proof.
\end{proof}
\begin{corollary} \label{cor:chi-bound}
    Let $h :\mathbb{R}^n \to [0,1]$ be any bounded measurable estimator. Then 
    \[ \left| \mathbb{E}_{Q_{\mu', \delta}}[h] - \mathbb{E}_{P_{\mu' + \eta\delta}}[h] \right| \le \sqrt{ \mathbb{E}_{P_{\mu' + \eta\delta}}[(h(X) - (\mu' + \eta\delta))^2] \cdot \chi^2(Q_{\mu', \delta} \| P_{\mu' + \eta\delta}) }. \]
\end{corollary}
\begin{proof}
    Let $P \coloneqq P_{\mu' + \eta\delta}$ and let $R(X) = h(X) - (\mu' + \eta\delta)$. Applying the Cauchy-Schwarz inequality,
    \begin{align*}
        \left| \mathbb{E}_{Q_{\mu', \delta}}[h] - \mathbb{E}_P[h] \right| &= \left| \mathbb{E}_P \left[ R(X) \left( \frac{dQ_{\mu', \delta}}{dP}(X) - 1 \right) \right] \right| \\
        &\le \sqrt{ \mathbb{E}_P [R(X)^2] \cdot \mathbb{E}_P \left[ \left( \frac{dQ_{\mu', \delta}}{dP}(X) - 1 \right)^2 \right] },
    \end{align*}
    which perfectly matches the desired statement.
\end{proof}

\subsection{Tracking the Global Shift}
Having established indistinguishability, we will now show that the estimator must track the global shift to maintain its uniform MSE accuracy assumption. Indeed, let $t \coloneqq \eta \delta$, where $\delta > 0$ is a small constant chosen later. We let $a(\mu')$ and $b(\mu')$ be the expected output and bias of $h$ respectively. 

To lower-bound the supremum over $\mu'$, we lower-bound the average over a carefully chosen prior. Let $w \colon \mathbb{R} \to \mathbb{R}_{+}$ be a continuously differentiable probability density function supported strictly on $[\epsilon, 1-\epsilon]$ for some fixed $\epsilon > 0$. We assume $t < \epsilon$ so that the shifted density $w(\cdot - t)$ remains entirely supported within $[0,1]$.

By our risk assumption, the average squared bias over the interval satisfies $\int_0^1 b(\mu')^2 \,d\mu' \le \int_0^1 \mathbb{E}_{X \sim P_{\mu'}}[(h(X)-\mu')^2] \,d\mu' \le C_{\mathsf{int}}/n$. Because $w$ is smooth and compactly supported, its derivative $w'$ has finite $L^2$ norm, ensuring $\|w(\cdot - t) - w(\cdot)\|_{L^2} \le t \|w'\|_{L^2}$. 
Therefore, by change of variables, we have 
\begin{align*}
    \int w(\mu') (a(\mu'+t) - a(\mu')) \, d\mu' &= t + \int b(v)(w(v-t) - w(v)) \, dv \\
    &\ge t - \|b\|_{L^2} \|w(\cdot - t) - w(\cdot)\|_{L^2} \\
    &\ge t - \sqrt{\frac{C_{\mathsf{int}}}{n}} \cdot t \|w'\|_{L^2} \\
    &= t - O\left(\frac{t}{\sqrt{n}}\right) \ge t - o(t).
\end{align*}
Thus, on average over the prior $w$, the estimator's mean output is mathematically forced to shift by approximately $t = \eta \delta$.

To finish the proof, notice that because $Q_{\mu', \delta}$ is generated by shifting exactly $k = \lfloor \eta n \rfloor$ coordinates of $X \sim P_{\mu'}$, the resulting dataset $X^{(I, \delta)}$ is a valid $\eta$-corruption. Therefore, we have $S_\eta^h(X) \ge |h(X^{(I, \delta)}) - h(X)|$. By Triangle Inequality, we have 
\begin{align*}
    \mathbb{E}_{X \sim P_{\mu'}}[S_\eta^h(X)] &\ge \mathbb{E}_{X \sim P_{\mu'}, I}[|h(X^{(I, \delta)}) - h(X)|] \\ 
    &\ge \mathbb{E}_{X \sim P_{\mu'}, I}[h(X^{(I, \delta)})] - \mathbb{E}_{X \sim P_{\mu'}}[h(X)] \\ 
    &\ge \mathbb{E}_{X \sim Q_{\mu', \delta}}[h(X)] - \mathbb{E}_{X \sim P_{\mu'}}[h(X)] \\ 
    &= (\mathbb{E}_{X \sim P_{\mu' + t}}[h(X)] - \mathbb{E}_{X \sim P_{\mu'}}[h(X)] ) + (\mathbb{E}_{X \sim Q_{\mu', \delta}}[h(X)] - \mathbb{E}_{X \sim P_{\mu' + t}}[h(X)]) \\ 
    &\ge a(\mu' + t) - a(\mu') - \left| \mathbb{E}_{X \sim Q_{\mu', \delta}}[h(X)] - \mathbb{E}_{X \sim P_{\mu' + t}}[h(X)] \right|
\end{align*}
We now average this inequality over the prior $w(\mu')$ to establish the final lower bound. The first term satisfies $\int w(\mu') (a(\mu'+t) - a(\mu')) \, d\mu' \ge \eta \delta - o(\eta \delta)$, where the $o(\eta \delta)$ term vanishes as $n \to \infty$ due to the smoothness of $w$. For the second term, we apply \cref{cor:chi-bound} and the Cauchy-Schwarz inequality over the prior integral to decouple the estimator's accuracy from the statistical distance between distributions:
\begin{align*}
    \int w(\mu') &\left| \mathbb{E}_{Q_{\mu', \delta}}[h] - \mathbb{E}_{P_{\mu' + t}}[h] \right| \, d\mu' \\
    &\le \int w(\mu') \sqrt{\mathbb{E}_{P_{\mu' + t}}[(h(X) - (\mu' + t))^2] \cdot \chi^2(Q_{\mu', \delta} \| P_{\mu' + t})} \, d\mu' \\
    &\le \sqrt{ \int w(\mu') \mathbb{E}_{P_{\mu' + t}}[(h(X) - (\mu' + t))^2] \, d\mu' } \cdot \sqrt{ \int w(\mu') \chi^2(Q_{\mu', \delta} \| P_{\mu' + t}) \, d\mu' }.
\end{align*}
We evaluate the two components under the square root independently:
\begin{enumerate}
    \item By the uniform MSE assumption on the scalar estimator $h$, we have that variance is bounded by MSE, and thus
    \[ \int w(\mu') \mathbb{E}_{P_{\mu' + t}}[(h(X) - (\mu' + t))^2] \, d\mu' \le \int w(v - t) \mathbb{E}_{P_{v}}[(h(X) - v)^2] \, dv \le \frac{C_{\mathsf{int}}}{n} \| w \|_{\infty}. \]
    \item  Applying Lemma \ref{lem:close-dist-corr} for the regime $k \le \kappa\sqrt{n}$, the divergence satisfies
    \[ \chi^2(Q_{\mu', \delta} \| P_{\mu' + t}) \le \exp\left( \frac{k^2}{n}(e^{\delta^2} - 1 - \delta^2) \right) - 1 \lesssim \frac{k^2 \delta^4}{n}. \]
\end{enumerate}
Substituting these bounds into the Cauchy-Schwarz product yields an error term of:
\[ \sqrt{ \frac{C_{\mathsf{int}}}{n} \| w \|_{\infty} \cdot \frac{k^2 \delta^4}{n} } = \sqrt{C_{\mathsf{int}} \| w \|_{\infty}} \cdot \frac{k \delta^2}{n} = O(\eta \delta^2). \]
Putting this together, the average empirical sensitivity is bounded by:
\[ \int w(\mu') \mathbb{E}_{X \sim P_{\mu'}}[S_\eta^h(X)] \, d\mu' \ge \eta \delta - o(\eta \delta) - O(\eta \delta^2). \]
By choosing the absolute constant $\delta > 0$ to be sufficiently small, the $O(\eta \delta^2)$ term is made strictly smaller than $\eta \delta / 2$. This guarantees an average expected sensitivity of $\Omega(\eta \delta) = \Omega(\eta)$. Since the average sensitivity is lower-bounded by $c_{\mathrm{low}} \eta$, the supremum over the interval must also satisfy
\[ \sup_{\mu' \in [0,1]} \mathbb{E}_{X \sim \mathcal{N}(\mu', 1)^{\otimes n}}[S_\eta^h(X)] \ge c_{\mathrm{low}} \eta. \]
This establishes the scalar claim, which transfers to the high-dimensional estimator $f$ via the orthogonal projection, completing the proof of \cref{thm:mean-low}. 
\section{Proof of~\texorpdfstring{\cref{thm:mean-high}}{}}
\label{app:mean-high}
Throughout this section, for a scalar parameter $\mu' \in \mathbb{R}$, we write $P_{\mu'} \coloneqq \mathcal{N}(\mu', 1)^{\otimes n}$. By the reduction established in \cref{lem:scalar}, to prove the high-dimensional theorem, it suffices to prove the following scalar claim: if the induced one-dimensional estimator $h \colon \mathbb{R}^n \to [0,1]$ satisfies the integrated risk bound
\[ \int_0^1 \mathbb{E}_{X \sim P_{\mu'}}[(h(X) - \mu')^2] \, d\mu' \le \frac{C_{\mathsf{int}}}{n}, \]
then, for a corruption budget $k = \lfloor \eta n \rfloor \ge C_{\mathrm{high}} \log n$, its average empirical sensitivity must be large:
\[ \sup_{\mu' \in [0,1]} \mathbb{E}_{X \sim P_{\mu'}}[S_\eta^h(X)] \ge c_{\mathrm{high}} \eta. \]

To see exactly why this scalar claim implies the full high-dimensional theorem, recall that \cref{lem:scalar} guarantees the existence of a "good" unit direction $u \in \mathbb{S}^{d-1}$ and an orthogonal shift $\lambda_{\ast} \in u^\perp$. When we evaluate the original high-dimensional estimator $f$ along the one-dimensional affine line parameterized by $\mu' u + \lambda_{\ast}$, its empirical sensitivity strictly upper-bounds the sensitivity of the induced scalar estimator $h$. Together with Jensen's inequality, this gives us 
\[ \ES_{\eta, 2}(f; \mu' u + \lambda_{\ast}) \ge \left( \mathbb{E}_{X \sim P_{\mu'}}[S_\eta^h(X)^2] \right)^{1/2} \ge \mathbb{E}_{X \sim P_{\mu'}}[S_\eta^h(X)]. \]
Taking the supremum over $\mu' \in [0,1]$ on both sides demonstrates that establishing the $c_{\mathrm{high}} \eta$ lower bound for $h$ immediately yields the required lower bound for $f$, completing the reduction.

We start with the following lemma about endpoint average bound for scalar estimator bound with integrated MSE bound. 
\begin{lemma} \label{lem:end-avg}
Assume $\eta \in \left( 0, \frac{1}{10} \right)$. If the scalar estimator $h \colon \mathbb{R}^n \to [0,1]$ satisfies the integrated MSE bound
\[ \int_0^1 \mathbb{E}_{X \sim P_{\mu'}}[(h(X) - \mu')^2] \ d\mu'  \le \frac{C_{\mathsf{int}}}{n} \] 
Assume moreover that $n$ is large enough such that $C_{\mathsf{high}} \log n \ge 100 C_{\mathsf{int}}$. Then we have the following endpoint average bounds: 
\[ \frac{1}{\eta} \int_0^{\eta} a(\mu') \ d \mu' \le \frac{1}{3} \ \text{and} \ \frac{1}{\eta} \int_{1-\eta}^1 a(\mu') \ d\mu' \ge \frac{2}{3}.  \]
\end{lemma}
\begin{proof}
By the integrated MSE assumption, 
\[ \frac{1}{\eta} \int_0^{\eta} \mathbb{E}_{X \sim P_{\mu'}} [(h(X) - \mu')^2] \ d\mu' \le \frac{C_{\mathsf{int}}}{\eta n} \le \frac{1}{100}, \]
by our choice of $n$ large enough so that $\eta n \ge k  \ge C_{\mathsf{high}} \log n \ge 100 C_{\mathsf{int}}$. Similarly, we obtain
\[ \frac{1}{\eta} \int_{1 - \eta}^1 \mathbb{E}_{X \sim P_{\mu'}}[(h(X) - \mu')^2] \ d\mu' \le \frac{1}{100}. \]
Now, by noticing that for $\mu' \in [0,\eta]$, we have $a(\mu') \le \mu' + \mathbb{E}_{X \sim P_{\mu'}}[|h(X) - \mu'|]$. Therefore,
\begin{align*}
    \frac{1}{\eta} \int_0^{\eta} a(\mu') \ d\mu' &\le \frac{1}{\eta} \int_0^{\eta} \mu' \ d\mu' + \frac{1}{\eta} \int_0^{\eta} \mathbb{E}_{X \sim P_{\mu'}}[|h(X) - \mu'|] \ d\mu ' \\ 
    &\le \frac{\eta}{2} + \left( \frac{1}{\eta} \int_0^{\eta} \mathbb{E}_{X \sim P_{\mu'}} [(h(X) - \mu')^2] \ d\mu' \right)^{1/2} \\ 
    &\le \frac{\eta}{2} + \frac{1}{10} \le \frac{1}{3}. 
\end{align*}
Similarly, for $\mu' \in [1 - \eta, 1]$, we have $a(\mu') \ge \mu' - \mathbb{E}_{X \sim P_{\mu'}}[|h(X) - \mu'|]$. Therefore, 
\begin{align*}
    \frac{1}{\eta} \int_{1 - \eta}^1 a(\mu') \ d\mu' &\ge \frac{1}{\eta} \int_{1 - \eta}^1 \mu' \ d\mu' - \frac{1}{\eta} \int_{1 - \eta}^1 \mathbb{E}_{X \sim P_{\mu'}}[|h(X) - \mu'|] \ d\mu' \\ 
    &\ge 1 - \frac{\eta}{2} - \left( \frac{1}{\eta} \int_{1 - \eta}^1 \mathbb{E}_{X \sim P_{\mu'}}[(h(X) - \mu')^2] \ d\mu' \right)^{1/2} \\ 
    &\ge 1 - \frac{\eta}{2} - \frac{1}{10} \ge \frac{2}{3}. 
\end{align*}
This proves the lemma. 
\end{proof}
\subsection{Average Mean Displacement}
Take $n$ sufficiently large such that the condition in \cref{lem:end-avg} holds. We first quantify the global displacement of the estimator's expected mean under a shift of $\eta$. Indeed, we have 
\begin{align*}
    \int_0^{1 - \eta} (a(\mu' + \eta) - a(\mu')) \ d\mu' &= \int_{\eta}^1 a(\mu') \ d\mu' - \int_0^{1 - \eta} a(\mu') \ d\mu' \\ 
    &= \int_{1 - \eta}^1 a(\mu') \ d\mu' - \int_0^{\eta} a(\mu') \ d\mu' \ge \frac{\eta}{3}. 
\end{align*}
Let $\mathcal{U}$ denote the uniform probability measure over the interval $[0,1-\eta]$. We have
\[ \mathbb{E}_{\mu' \sim \mathcal{U}}[a(\mu' + \eta) - a(\mu')] = \frac{1}{1 - \eta} \int_0^{1 - \eta} (a(\mu' + \eta) - a(\mu')) \ d\mu' \ge \frac{\eta}{3}. \]

\subsection{Coordinate-wise Maximal Coupling}
For any fixed $\mu' \in [0,1-\eta]$, we construct a joint distribution over $(X, X')$ via independent coordinatewise maximal couplings as in \cref{thm:maximal-coupling} of the marginals $\mathcal{N}(\mu', 1)$ and $\mathcal{N}(\mu' + \eta, 1)$ respectively. In this case, $X \sim P_{\mu'}$ and $X' \sim P_{\mu' + \eta}$. Let $D \coloneqq \# \{ i \in [n] \colon X_i \not= X_i' \}$ be the number of coordinates where the samples differ. By \cref{thm:maximal-coupling}, $D \sim \operatorname{Bin}(n, p)$ where the success probability $p$ satisfies $p < \eta$. Let $k \coloneqq \lfloor \eta n \rfloor$, we see that there exists an absolute constant $\rho \in (0,1)$ such that $\mathbb{E}[D] \le \rho k$ for all $\eta \in (0,1/10]$. Therefore, a standard Chernoff bound argument as in \cref{thm:chernoff} gives us a constant $c_{\mathsf{Ch}} > 0$ such that $\mathbb{P}(D > k) \le \exp(-c_{\mathsf{Ch}} k)$. We may choose $n$ and $C_{\mathsf{high}}$ large enough such that in the regime $k \ge C_{\mathsf{high}} \log n$, we are guaranteed that $\mathbb{P}(D > k) \le \eta/6$ for all $n$. 

On the event $\{ D \le k \}$, the datasets $X$ and $X'$ differ by at most $k$ points. Therefore, $X'$ is a valid $\eta$-corruption of $X$. By definition, $S_\eta^h(X) \ge |h(X') - h(X)| \cdot \mathbf{1} \{ D \le k \}$. Since $h \in [0,1]$, we have $|h(X') - h(X)| \le 1$. Therefore, 
\begin{align*}
    \mathbb{E}_{X \sim P_{\mu'}}[S_\eta^h(X)] &\ge \mathbb{E} \left[ |h(X') - h(X)| \cdot \mathbf{1} \{ D \le k \} \right]  \\ 
    &\ge \mathbb{E}[(h(X') - h(X)) \cdot \mathbf{1} \{ D \le k \}] \\ 
    &= \mathbb{E}[h(X') - h(X)] - \mathbb{E}[(h(X') - h(X)) \cdot \mathbf{1} \{ D > k \}] \\ 
    &\ge a(\mu' + \eta) - a(\mu') - \mathbb{P}(D > k).
\end{align*}
To conclude the proof, we average this final inequality over the uniform prior $\mathcal{U}$. This gives us 
\begin{align*}
    \mathbb{E}_{\mu' \sim \mathcal{U}} \mathbb{E}_{X \sim P_{\mu'}}[S_\eta^h(X)] &\ge \mathbb{E}_{\mu' \sim \mathcal{U}}[a(\mu' + \eta) - a(\mu')] - \mathbb{P}(D > k) \\ 
    &\ge \frac{\eta}{3} - \frac{\eta}{6} = \frac{\eta}{6}. 
\end{align*}
This implies $\sup_{\mu' \in [0,1]} \mathbb{E}_{X \sim P_{\mu'}}[S_\eta^h(X)] \ge \frac{\eta}{6}$, which is what we wanted. By \cref{lem:scalar}, we thus have $\sup_{\mu \in \mathbb{R}^d} \ES_{\eta, 2}(f; \mu) \ge \frac{\eta}{6}$, which proves \cref{thm:mean-high}. 

\section{Proof of~\texorpdfstring{\cref{thm:variance-main}}{}}
\label{app:variance-main}
The proof of \cref{thm:variance-main} can be split into two parts. First, we show that uniform MSE accuracy forces the estimator to have output variance $\Omega(d/n)$ at some parameter point. Then we convert this clean-sample variance into a corruption strategy via block resampling. 
\subsection{Existence of large variance parameter}
\label{app:large-variance}
We first prove the statistical part of the variance obstruction. To do this, we give two proofs. The first is the shorter Cramer-Rao proof, which assumes a regularity condition on the mean map. The second is the Hammersley-Chapman-Robbins proof, which may be applied to broader distribution classes, without using similar regularity assumptions.

Throughout, for $X = (X_1, \dots, X_n) \sim \mathcal{N}(\mu, \mathbb{I}_d)^{\otimes n}$, we write
\[ \Var_{\mu}(f) \coloneqq \Var_{\mu}(f(X)) \coloneqq \mathbb{E}_{X} \| f(X) - \mathbb{E}_{X}f(X) \|^2. \]

\begin{proposition} \label{prop:large-variance}
Let $f \colon (\mathbb{R}^d)^n \to \mathbb{R}^d$ be measurable and suppose that
\[ \sup_{\mu \in \mathbb{R}^d} \mathbb{E}_{X} \| f(X) - \mu \|^2 \le C_{\mathsf{MSE}} \frac{d}{n}. \]
Then there exists a universal constant $c_{\mathrm{var}} > 0$ and a parameter $\mu^{\ast} \in \mathbb{R}^d$ such that
\[ \Var_{\mu^{\ast}}(f) \ge c_{\mathrm{var}} \frac{d}{n}. \]
\end{proposition}
We will first show two different proofs for the scalar version of \cref{prop:large-variance}, then we will use the same argument to lift this to the $d$-dimensional version. Indeed, for the scalar version: consider $g = g(Z_1, \dots, Z_n)$ under $Z_1, \dots, Z_n \stackrel{\text{iid}}{\sim} \mathcal{N}(\mu', 1)$, where we write $a(\mu') \coloneqq \mathbb{E}_{Z}[g(Z)]$ and $\Var_{\mu'}(g) \coloneqq \mathbb{E}_{Z}[(g - a(\mu'))^2]$.

\subsubsection{Proof I of Scalar Version of \cref{prop:large-variance}}
This proof is the Cramer-Rao proof, which requires the standard regularity condition that the mean map $a(\mu')$ is absolutely continuous and satisfies the Gaussian score identity.

Define the bias $b(\mu')  := a(\mu') - \mu'$. By Jensen's inequality and the uniform MSE accuracy bound, we have
\[ |b(\mu')| = |\mathbb{E}_Z[g(Z) - \mu']| \le \left(\mathbb{E}_Z[(g(Z) - \mu')^2] \right)^{1/2} \le \sqrt{\frac{C_{\mathsf{MSE}}}{n}}\]
for every $\mu'$. Therefore, for $L \coloneqq 4 \max \left \{ \frac{1}{\sqrt{n}}, \sqrt{\frac{C_{\mathsf{MSE}}}{n}} \right \}$, we have 
\[ a(L) - a(0) = L + b(L) - b(0) \ge L - 2 \sqrt{\frac{C_{\mathsf{MSE}}}{n}} \ge \frac{L}{2}. \]
By \cref{thm:cramer-rao}, applied to the Gaussian location family
$P_{\mu'}=\mathcal N(\mu',1)^{\otimes n}$ and the statistic $T=g$, the Fisher information is
$I(\mu')=n$. Therefore, under the stated regularity assumptions on the mean map
$a(\mu')=\mathbb E_{\mu'}g$, we have
\[
    \Var_{\mu'}(g)
    \ge
    \frac{(a'(\mu'))^2}{n}.
\]
Averaging over $[0,L]$, we obtain 
\begin{align*}
    \frac{1}{L} \int_0^L \Var_{\mu'}(g) \ d \mu' &\ge \frac{1}{nL} \int_0^L a'(\mu')^2 \ d\mu' \\ 
    &\ge \frac{1}{n} \left( \frac{1}{L} \int_0^L a'(\mu') \ d \mu' \right)^2 \\ 
    &= \frac{1}{n} \left( \frac{a(L) - a(0)}{L} \right)^2 \ge \frac{1}{4n}. 
\end{align*}
This proves the claim for $d = 1$. 
\subsubsection{Proof II of Scalar Version of \cref{prop:large-variance}}
We now give the Hammersley-Chapman-Robbins proof. This proof avoids any differentiability assumptions. Let us consider $h \coloneqq \frac{1}{\sqrt{n}}$, $L \coloneqq 4 \max \left \{ \sqrt{\frac{C_{\mathsf{MSE}}}{n}}, h \right \}$ and $J \coloneqq \left \lceil \frac{L}{h} \right \rceil$. 

Now, we consider the grid points $\mu_{\ell}' \coloneqq \ell h$ for $\ell = 0, \dots, J$. Since $\mu_J' = Jh \ge L$ and $|b(\mu')| \le \sqrt{\frac{C_{\mathsf{MSE}}}{n}}$, it follows that
\[ a(\mu_J') - a(\mu_0') = \mu_J' - \mu_0' + b(\mu_J') - b(\mu_0') \ge L - 2 \sqrt{\frac{C_{\mathsf{MSE}}}{n}} \ge \frac{L}{2}. \]
For each $\ell = 0, \dots, J - 1$, we let $P_{\ell}$ and $Q_{\ell}$ denote the probability distribution $\mathcal{N}(\mu_{\ell}', 1)^{\otimes n}$ and $\mathcal{N}(\mu_{\ell + 1}', 1)^{\otimes n}$ respectively. Because $\mu_{\ell + 1}' - \mu_{\ell}' = h = \frac{1}{\sqrt{n}}$, we have 
\[ \chi^2(Q_{\ell} \| P_{\ell}) = \exp(nh^2) - 1 = e - 1. \]
Applying \cref{thm:hcr} with $P = P_{\ell} = \mathcal{N}(\mu_{\ell}', 1)^{\otimes n}, Q = Q_{\ell} = \mathcal{N}(\mu_{\ell + 1}', 1)^{\otimes n}$ and $T = g$, we obtain
\[ \Var_{\mu_{\ell}'}(g) \ge \frac{(a(\mu_{\ell + 1}') - a(\mu_{\ell}'))^2}{e - 1}. \]
Summing over $\ell = 0, \dots, J - 1$ and apply Cauchy Schwarz gives us 
\[ \sum_{\ell = 0}^{J - 1} \Var_{\mu_{\ell}'}(g) \ge \frac{1}{e - 1} \sum_{\ell = 0}^{J - 1} (a(\mu_{\ell + 1}') - a(\mu_{\ell}'))^2 \ge \frac{1}{(e -  1)J} (a(\mu_J') - a(\mu_0'))^2 \ge \frac{L^2}{4(e - 1) J}. \]
Consequently, the average variance satisfies
\[ \frac{1}{J + 1} \sum_{\ell = 0}^{J} \Var_{\mu_{\ell}'}(g) \ge \frac{L^2}{4(e - 1) J(J + 1)} \ge \frac{2}{15(e - 1)} \cdot \frac{1}{n} \]
This proves the claim for $d = 1$. 
\subsubsection{Proof of \cref{prop:large-variance}}
We will now state the following general lemma that allows one to lift our argument from one dimension to higher dimension, as long as one can get an averaging measure over scalar parameters under which the expected output variance is $\Omega(1/n)$.

\begin{lemma}
\label{lem:scalar-averaging-principle}
Fix $C_{\mathsf{MSE}} > 0$ and set $C_0 := C_{\mathsf{MSE}} d$. Suppose there is a probability measure $\nu_{C_0}$ on $\mathbb{R}$ and a universal constant $c_{\mathsf{sc}} > 0$ with the following property: For every auxiliary random element $W$ whose law does not depend on $\theta$, and every real-valued statistics $g = g(Z_1, \dots, Z_n, W)$, with $Z_1, \dots, Z_n \stackrel{\text{iid}}{\sim} \mathcal{N}(\theta, 1)$ independent of $W$, then 
\[ \sup_{\theta \in \mathbb{R}} \mathbb{E}[(g - \theta)^2] \le \frac{C_0}{n} \implies \int \Var_\theta(g) \ d\nu_{C_0}(\theta) \ge \frac{c_{\mathsf{sc}}}{n} \]
holds, where the variance is over both $Z_1, \dots, Z_n$ and $W$. Then every measurable estimator $f: (\mathbb{R}^d)^n \to \mathbb{R}^d$ satisfying
\[ \sup_{\mu \in \mathbb{R}^d} \mathbb{E} \| f(X) - \mu \|_2^2 \le C_{\mathsf{MSE}} \frac{d}{n} \]
has a parameter $\mu^{\ast} \in \mathbb{R}^d$ such that $\Var_{\mu^{\ast}}(f) \ge c_{\mathsf{sc}} \frac{d}{n}$. 

\end{lemma}

\begin{proof}
Let $C_0 \coloneqq C_{\mathsf{MSE}}d$, and let $\nu \coloneqq \nu_{C_0}$ be the scalar averaging measure supplied by the assumption. We average the parameters $\mu\in\mathbb{R}^d$ under the product measure $\nu^{\otimes d}$. 

Fix a coordinate $j\in[d]$, and fix the remaining coordinates $\mu_{-j} = (\mu_1,\dots,\mu_{j-1},\mu_{j+1},\dots,\mu_d) \in \mathbb{R}^{d-1}$. Consider the scalar submodel $\mu(\theta) \coloneqq (\mu_1,\dots,\mu_{j-1},\theta,\mu_{j+1},\dots,\mu_d)$ for $\theta\in\mathbb{R}$. For $X\sim\mathcal{N}(\mu(\theta),I_d)^{\otimes n}$, let $Z_i \coloneqq X_{i,j}$ for $i\in[n]$, and let $W \coloneqq (X_{i,\ell})_{1\le i\le n,\ \ell\ne j}$ denote all other coordinates of the sample. Then $Z_1,\dots,Z_n\stackrel{\mathrm{iid}}{\sim}\mathcal{N}(\theta,1)$ and $W$ is independent of $Z$. Crucially, because $\mu_{-j}$ is fixed, the law of $W$ does not depend on $\theta$.

Define the scalar statistic $g_{j,\mu_{-j}}(Z_1,\dots,Z_n,W) \coloneqq f_j(X)$. For every $\theta\in\mathbb{R}$, the uniform MSE assumption implies that the error of this single coordinate is bounded by the total vector error: $\mathbb{E}[(g_{j,\mu_{-j}}-\theta)^2] \le \mathbb{E} \|f(X)-\mu(\theta)\|_2^2 \le C_{\mathsf{MSE}}\frac{d}{n} = \frac{C_0}{n}$. Therefore, the scalar averaging principle applies, yielding $\int \mathrm{Var}_{\mu(\theta)}(f_j(X)) \,d\nu(\theta) \ge c_{\mathrm{sc}}/n$ for every fixed $\mu_{-j}$. Averaging this inequality over $\mu_{-j}\sim\nu^{\otimes(d-1)}$ via Fubini's theorem yields $\mathbb{E}_{\mu\sim\nu^{\otimes d}} [\mathrm{Var}_{\mu}(f_j(X))] \ge c_{\mathrm{sc}}/n$. 

Summing over all $d$ coordinates gives the total expected variance:
\[ \mathbb{E}_{\mu\sim\nu^{\otimes d}} [\mathrm{Var}_{\mu}(f)] = \sum_{j=1}^d \mathbb{E}_{\mu\sim\nu^{\otimes d}} [\mathrm{Var}_{\mu}(f_j(X))] \ge c_{\mathrm{sc}}\frac{d}{n}. \]
Since the expected variance under $\mu\sim\nu^{\otimes d}$ is at least $c_{\mathrm{sc}}d/n$, there must exist at least one parameter $\mu^\ast\in\mathbb{R}^d$ such that $\mathrm{Var}_{\mu^\ast}(f) \ge c_{\mathrm{sc}}d/n$.
\end{proof}

To finish this, we may instantiate the above lifting principle theorem with two different scalar proof that we have established. The Cramer-Rao scalar proof proves the scalar averaging principle using the continuous measure $\nu_{C_0} = \mathrm{Unif}([0,L])$, $c_{\mathrm{sc}} = \frac{1}{4}$, where $L = 4 \max \{ \sqrt{C_0/n}, 1/\sqrt{n} \}$, provided that the scalar mean map satisfies the Cramer Rao regularity assumptions. 

Alternatively, the Hammersley--Chapman--Robbins scalar proof establishes the same principle using the discrete measure $\nu_{C_0} = \mathrm{Unif}\{0,h,2h,\dots,Jh\}$ and $c_{\mathrm{sc}} = \frac{2}{15(e-1)}$, where $h \coloneqq 1/\sqrt{n}$, $L = 4\max\{\sqrt{C_0/n}, h\}$, and $J \coloneqq \lceil L/h \rceil$.

This fully establishes \cref{prop:large-variance}.
\subsection{Conversion into Corruption Strategy}
\label{app:variance-corr}
Fix a parameter $\mu^{\ast}$ from \cref{prop:large-variance}, and let $X = (X_1, \dots, X_n) \sim \mathcal{N}(\mu^{\ast}, \mathbb{I}_d)^{\otimes n}$. Let $k = \lfloor \eta n \rfloor \ge 1$. Partition $[n]$ into $M \coloneqq \lceil n/k \rceil$ disjoint nonempty blocks $B_1, \dots, B_M$ of sizes at most $k$. Let $Y_i = (X_{\ell})_{\ell \in B_i}$ denote the data in block $B_i$. Equivalently, we may view $X$ as the collection $(Y_1, \dots, Y_M)$. Let $Y_i'$ be an independent fresh copy of $Y_i$, independent of all other variables, and let $X^{(i)}$ denote the dataset obtained from $X$ by replacing block $Y_i$ by $Y_i'$. 

We note that by definition, each $X^{(i)}$ differs from $X$ in at most $k$ sample positions. Therefore, pointwise, we have
\[ S_\eta^f(X)^2 \ge \| f(X) - f(X^{(i)}) \|^2, \quad \forall i = 1, 2, \dots, M. \]
By vector Efron-Stein inequality in \cref{thm:efron-stein}: for independent blocks $Y_1, \dots, Y_M$ and independent replacements $Y_i'$, we have
\[ \Var_{\mu^{\ast}}(f(X)) \le \frac{1}{2} \sum_{i = 1}^M \mathbb{E} \| f(X) - f(X^{(i)}) \|^2. \]
Hence, we see that
\[ \max_{1 \le i \le M} \mathbb{E} \| f(X) - f(X^{(i)} ) \|^2 \ge \frac{2}{M} \Var_{\mu^{\ast}}(f(X)). \]
Combining this with \cref{prop:large-variance}, we see that there exists some constant $c_0 > 0$ such that
\[ \mathbb{E}_{X}[S_{\eta}^{f}(X)^2] \ge c_0 \frac{\eta d}{n}. \]
Therefore, we have 
\[ \sup_{\mu \in \mathbb{R}^d} \ES_{\eta, 2}(f; \mu) = \sup_{\mu \in \mathbb{R}^d} \left( \mathbb{E}_X[S_\eta^f(X)^2] \right)^{1/2} \ge \sqrt{c_0} \sqrt{\frac{\eta d}{n}}, \]
and taking $\useconstant{const:varLower} \coloneqq \sqrt{c_0}$ proves \cref{thm:variance-main}.

\section{Proof of~\texorpdfstring{\cref{thm:bernoulli-main}}{}}
\label{app:bernoulli-main}
Throughout this section, for $t \in \{ 0, \dots, n \}$, define the Hamming-weight layer $\mathsf{Layer}(t) \coloneqq \{ x \in \{ 0, 1 \}^n : |x| = t \}$. Let $U_t$ denote a uniformly random element of $\mathsf{Layer}(t)$. 

\begin{lemma} \label{lem:beta-binom}
Let $P \sim \mathsf{Unif}[0,1]$ and, conditional on $P = p$, sample $X \sim \operatorname{Bern}(p)^{\otimes n}$. Let $T \coloneqq |X| = \sum_{i =1}^n X_i$. Then $\mathbb{P}(T = t) = \frac{1}{n + 1}$ for all $t \in \{ 0, \dots, n \}$. Therefore, conditional on $T = t$, $X$ is uniform on $\mathsf{Layer}(t)$. 
\end{lemma}
\begin{proof}
    For $t \in \{ 0, \dots, n \}$, we have 
    \[ \mathbb{P}(T = t) = \int_0^1 \binom{n}{t} p^t ( 1-p)^{n - t} \ dp = \binom{n}{t} \frac{t!(n - t)!}{(n + 1)!} = \frac{1}{n + 1}. \]
\end{proof}

\begin{lemma} \label{lem:layer-transport}
 Fix integers $t, \ell$ with $0 \le t \le n - \ell$. Sample $U_t$ uniformly from $\mathsf{Layer}(t)$. Conditional on $U_t = x$, choose an $\ell$-subset $I$ uniformly among the zero coordinates of $x$ and flip those coordinates to $1$. Let $V$ be the resulting vector. Then $V \sim \mathsf{Unif}(\mathsf{Layer}(t + \ell))$ and $d_H(U_t, V) = \ell$. 
\end{lemma}
\begin{proof}
    The Hamming-distance claim is immediate because exactly $\ell$ zero coordinates are flipped to one. It remains to prove uniformity. Fix $y \in \mathsf{Layer}(t + \ell)$. A pair $(x, I)$ maps to $y$ exactly when $I$ is an $\ell$-subset of the one coordinates of $y$, and $x$ is obtained from $y$ by changing the coordinates in $I$ back to zero. Thus, every $y \in \mathsf{Layer}(t + \ell)$ has exactly $\binom{t + \ell}{\ell}$ preimages. Since the sampling probability of each valid pair $(x, I)$ is the same, $V$ is thus uniform on $\mathsf{Layer}(t + \ell)$. 
\end{proof}

We are now ready to prove \cref{thm:bernoulli-main}. Indeed, we may replace $f = \operatorname{clip}_{[0,1]}\circ f$ by \cref{lem:clipping}. 

Let $\varepsilon \coloneqq \frac{C_{\mathsf{Bern}}}{\sqrt{n}}$. Choose a constant $\alpha=\alpha(C_{\mathsf{Bern}})\in(0,1/4]$ small enough that $2\alpha + \frac{2C_{\mathsf{Bern}}}{c_*}\sqrt{\alpha} \le \frac{1}{2}$. Let $k\coloneqq\lfloor\eta n\rfloor$ and define $\ell\coloneqq\min\{k,\lfloor\alpha n\rfloor\}$. For all sufficiently large $n$, depending only on $\alpha$, the assumptions $k\ge1$ and $n$ sufficiently large imply $1\le\ell\le k$ and $\ell\le\frac{n}{2}$. 

For each $t \in \{0, \dots, n\}$, define $a_t \coloneqq \mathbb E[f(U_t)]$. Fix $t\in\{0,\dots,n-\ell\}$. By \cref{lem:layer-transport}, we can couple $U_t$ and $U_{t+\ell}$ so that $d_H(U_t,U_{t+\ell})=\ell\le k$. Hence $\mathbb E[S_\eta^f(U_t)] \ge \mathbb E[|f(U_{t+\ell})-f(U_t)|] \ge |a_{t+\ell}-a_t|$. Therefore, we have 
\[\sum_{t=0}^{n-\ell} \mathbb E[S_\eta^f(U_t)] \ge \sum_{t=0}^{n-\ell}|a_{t+\ell}-a_t|.\] 
Decompose RHS into arithmetic progressions modulo $\ell$. For each residue $r\in\{0,\dots,\ell-1\}$, write $r, r+\ell, \dots, r+T_r\ell$, where $T_r\coloneqq\lfloor\frac{n-r}{\ell}\rfloor$ and $j_r\coloneqq r+T_r\ell$. Then $j_r\in\{n-\ell+1,\dots,n\}$. By the triangle inequality along this arithmetic progression, $\sum_{q=0}^{T_r-1} |a_{r+(q+1)\ell}-a_{r+q\ell}| \ge |a_{j_r}-a_r|$. Summing over $r=0,\dots,\ell-1$ yields 
\[ \sum_{t=0}^{n-\ell}|a_{t+\ell}-a_t| \ge \sum_{r=0}^{\ell-1}|a_{j_r}-a_r|. \] 

We now lower-bound the boundary gaps. Fix $r \in \{ 0, \dots, \ell - 1 \}$ and set $p = \frac{r}{n}$. If $X \sim \operatorname{Bern}(p)^{\otimes n}$ and $T = |X|$, then conditional on $T = r$, $X$ is uniform on $\mathsf{Layer}(r)$. Therefore, 
\[ \varepsilon \ge \mathbb{E}[|f(X) - p|] \ge \mathbb{P}(T = r) \mathbb{E}[|f(U_r) - p|] \ge \mathbb{P}(T = r) |a_r - p|. \]
We know that since $T \sim \mathsf{Bin}(n,p)$, \cref{lem:binomial-point-mass} gives us $\mathbb{P}(T = r) \ge \frac{c_{\ast}}{\sqrt{r + 1}} \ge \frac{c_{\ast}}{\sqrt{\ell}}$ for some positive constant $c_{\ast}$. Therefore, $|a_r - r/n| \le \frac{\varepsilon}{c_{\ast}} \sqrt{\ell}$. Since $r \le \ell - 1$, we have $a_r \le \frac{\ell}{n} + \frac{\varepsilon}{c_{\ast}} \sqrt{\ell}$. Similarly, fix $j \in \{ n - \ell + 1, \dots, n \}$ and write $j = n - s$ with $0 \le s \le \ell - 1$. Set $p = 1 - \frac{s}{n}$. By the same argument, we have $|a_j - j/n| \le \frac{\varepsilon}{c_{\ast}} \sqrt{\ell}$, and thus we get $a_j \ge 1 - \frac{\ell}{n} - \frac{\varepsilon}{c_{\ast}} \sqrt{\ell}$.  

Combining both of these, we see that for every $r \in \{ 0, \dots, \ell - 1 \}$, $|a_{j_r} - a_r| \ge 1 - \frac{2\ell}{n} - \frac{2 \varepsilon}{c_{\ast}} \sqrt{\ell}$. Because $\ell \le \alpha n$ and $\varepsilon = \frac{C_{\mathsf{Bern}}}{\sqrt{n}}$, our choice of $\alpha$ guarantees that $|a_{j_r} - a_r| \ge \frac{1}{2}$ for all residues $r$. Substituting this back into our telescoping sums give us
\[ \sum_{t = 0}^{n - \ell} \mathbb{E}[S_\eta^f(U_t)] \ge \frac{\ell}{2}. \]
To conclude, let $P \sim \mathsf{Unif}[0,1]$ and, conditional on $P = p$, sample $X \sim \operatorname{Bern}(p)^{\otimes n}$. By \cref{lem:beta-binom}, 
\[ \int_0^1 \mathbb{E}_{X \sim \operatorname{Bern}(p)^{\otimes n}}[S_\eta^f(X)] \ dp = \frac{1}{n + 1} \sum_{t = 0}^n \mathbb{E}[S_\eta^f(U_t)] \ge \frac{\ell}{2(n + 1)}. \]
Therefore this also gives us
\[ \sup_{p \in [0,1]}\mathbb{E}_{X \sim \operatorname{Bern}(p)^{\otimes n}}[S_\eta^f(X)] \ge \frac{\ell}{2(n + 1)}. \]
It remains to compare $\frac{\ell}{n + 1}$ to $\eta$. Indeed, suppose $k \le \lfloor \alpha n \rfloor$. Then $\ell = k$. In this case, $\frac{k}{n} \ge \frac{\eta}{2}$, and thus $\frac{\ell}{2(n + 1)} \ge \frac{\eta}{8}$. Otherwise, in the case $k > \lfloor \alpha n \rfloor$, then $\ell = \lfloor \alpha n \rfloor$. For all sufficiently large $n$, we have $\ell \ge \frac{\alpha n}{2}$ which therefore gives us the bound $\frac{\ell}{2(n + 1)} \ge \frac{\alpha}{8} \ge \frac{\alpha}{8} \eta$. This shows us
\[ \sup_{p \in [0,1]}\mathbb{E}_{X \sim \operatorname{Bern}(p)^{\otimes n}}[S_\eta^f(X)] \ge \frac{\alpha}{8} \eta. \]
Therefore, this theorem holds with $\useconstant{const:bernoulli} \coloneqq \frac{\alpha}{8}$.

\end{document}